\documentclass[a4paper,
oneside,
reqno,11pt]{amsart}

\usepackage[
margin=33.5mm,
rmargin=25mm,
lmargin=25mm,
marginparwidth=20mm,     % Width marginpar
marginparsep=2.5mm,       % Gap between text block and marginpar
]{geometry}

\usepackage[dvipsnames]{xcolor}
\usepackage{
	amssymb,
	amsmath,
	amsthm,
	eucal,
	dsfont,
	multicol,
	mathrsfs,
	tikz,
	graphicx,
	enumerate,
	amsfonts,
	latexsym,
	enumerate,
	enumitem,
	verbatim,
	yfonts
}

\usepackage{marginnote}

\usepackage{libertine}
\usepackage[libertine]{newtxmath}

\usepackage{cite}
\usepackage[
colorlinks=true,
linkcolor=magenta,
citecolor=magenta,
urlcolor=magenta,
pagebackref,
]{hyperref}

\numberwithin{equation}{section}

\let\oldtocsection=\tocsection
\let\oldtocsubsection=\tocsubsection
\let\oldtocsubsubsection=\tocsubsubsection

\renewcommand{\tocsection}[2]{\hspace{0em}\oldtocsection{#1}{#2}}
\renewcommand{\tocsubsection}[2]{\hspace{1em}\oldtocsubsection{#1}{#2}}
\renewcommand{\tocsubsubsection}[2]{\hspace{2em}\oldtocsubsubsection{#1}{#2}}

\usepackage{lipsum}

\makeatletter
\renewcommand*{\eqref}[1]{%
	\hyperref[{#1}]{\textup{\tagform@{\!\!\ref*{#1}}}}%
}%\makeatother %add when putting this file to arXiv

\makeatletter

\@addtoreset{equation}{section}
\makeatother
\makeatletter
\newcommand{\opnorm}{\@ifstar\@opnorms\@opnorm}
\newcommand{\@opnorms}[1]{%
	\left|\mkern-1.5mu\left|\mkern-1.5mu\left|
	#1
	\right|\mkern-1.5mu\right|\mkern-1.5mu\right|
}
\newcommand{\@opnorm}[2][]{%
	\mathopen{#1|\mkern-1.5mu#1|\mkern-1.5mu#1|}
	#2
	\mathclose{#1|\mkern-1.5mu#1|\mkern-1.5mu#1|}
}
\makeatother\theoremstyle{plain}
\newtheorem{theorem}{Theorem}[section]
\newtheorem{proposition}[theorem]{Proposition}
\newtheorem{lemma}[theorem]{Lemma}

\theoremstyle{definition}

\newtheorem{remark}[theorem]{Remark}

\def\L{\mathcal{L}}
\def\grad{\nabla_{\mathbb H}}
\def\Kor{\left|(z,t)\right|_{\mathbb H}}
\newcommand{\jap}[1]{\left\langle#1\right\rangle}
\newcommand{\norm}[1]{\left\lVert#1\right\rVert}%{{\|#1\|}}

\def\til{~}
\def\supp{\mathop{\mathrm{supp}}\nolimits}

\def\Id{\mathop{\mathrm{Id}}\nolimits}

\def\Ran{\mathop{\mathrm{Ran}}\nolimits}
\def\Ker{\mathop{\mathrm{Ker}}\nolimits}

\def\Re{\mathop{\mathrm{Re}}\nolimits}
\def\Im{\mathop{\mathrm{Im}}\nolimits}
\def\loc{\mathop{\mathrm{loc}}\nolimits}

\def\sgn{\mathop{\mathrm{sgn}}\nolimits}

\def\R{{\mathbb{R}}}

\def\N{{\mathbb{N}}}
\def\C{{\mathbb{C}}}
\def\Heis{{\mathbb{H}}}
\def\S_\sphere{{\mathbb{S}}}
\def\S{{\mathcal{S}}}

\def\F{{\mathcal{F}}}
\def\H{{\mathcal{H}}}

\def\<{{\langle}}
\def\>{{\rangle}}

\def\ep{{\varepsilon}}

\newcommand{\red}[1]{{\color{red}#1}}
\newcommand{\blue}[1]{{\color{blue}#1}}

\author[L.\ Fanelli]{Luca Fanelli}
\address[Luca Fanelli]{
	Ikerbasque, Basque Foundation for Science,
	48011 Bilbao, Spain,
	\newline \phantom{\quad} \&
	Universidad del Pa\'is Vasco / Euskal Herriko Unibertsitatea,
	48080 Bilbao, Spain,
	\newline \phantom{\quad}\&
	BCAM -- Basque Center for Applied Mathematics,
	48009 Bilbao, Spain}
\email{\href{mailto:luca.fanelli@ehu.es}{luca.fanelli@ehu.eus}}

\author[H.\ Mizutani]{Haruya Mizutani}
\address[Haruya Mizutani]{
	Department of Mathematics, Graduate School of Science, Osaka University, 
	\newline \phantom{\quad} Toyonaka, 560-0043 Osaka, Japan}
\email{\href{mailto:haruya@math.sci.osaka-u.ac.jp}{haruya@math.sci.osaka-u.ac.jp}}

\author[L.\ Roncal]{Luz Roncal}
\address[Luz Roncal]{
	BCAM -- Basque Center for Applied Mathematics,
	48009 Bilbao, Spain,
	\newline \phantom{\quad} \&
	Ikerbasque, Basque Foundation for Science,
	48011 Bilbao, Spain,
	\newline \phantom{\quad} \&
	Universidad del Pa\'is Vasco / Euskal Herriko Unibertsitatea,
	48080 Bilbao, Spain}
\email{\href{mailto:lroncal@bcamath.org}{lroncal@bcamath.org}}

\author[N. M. Schiavone]{Nico Michele Schiavone}
\address[Nico M. Schiavone]{
	Departamento de Matem\'atica e Inform\'atica Aplicadas a la Ingenier\'{i}a Civil y Naval, 
	\newline \phantom{\quad} 	Universidad Polit\'ecnica de Madrid, 
	\newline \phantom{\quad} 		28040 Madrid, Spain}
\email{\href{mailto:nico.schiavone@upm.es}{nico.schiavone@upm.es}}

\makeatletter
\@namedef{subjclassname@2020}{%
	\textup{2020} Mathematics Subject Classification}
\makeatother

\date{\today}

\title[Uniform resolvent estimates for the sublaplacian]{Uniform resolvent estimates, smoothing effects and spectral stability for the Heisenberg sublaplacian}

\begin{document}

\maketitle

\begin{abstract}
%We prove several uniform resolvent estimates for the sublaplacian and its fractional power on the Heisenberg group. The proof is based on Hardy's type inequalities due to Garofalo--Lanconelli and D'Ambrosio, and the method of weakly conjugate operator due to Hoshiro and Boutet de Monvel--M\u{a}ntoiu. As applications, we also obtain some global-in-time smoothing effects for the time-dependent Schr\"odinger equation, as well as the spectral stability for the Schr\"odinger operator with complex-valued potentials on the Heisenberg group.
%
We establish global bounds for solutions to stationary and time-dependent Schr\"odinger equations associated with the sublaplacian $\mathcal L$ on the Heisenberg group, as well as its pure fractional power $\mathcal L^s$ and conformally invariant fractional power $\mathcal L_s$. The main ingredient is a new abstract uniform weighted resolvent estimate which is proved by using the method of weakly conjugate operators --a variant of Mourre's commutator method-- and Hardy's type inequalities on the Heisenberg group. As applications, we show 
Kato-type smoothing effects for the time-dependent Schr\"odinger equation, 
and spectral stability of the sublaplacian  perturbed by complex-valued decaying potentials satisfying an explicit subordination condition. In the local case $s=1$, we obtain uniform estimates without any symmetry or derivative loss, which improve previous results. 
%
%for the case with the sublaplacian that global-in-time weighted space-time estimates with general initial data and Kato-type smoothing effects under the radial symmetry in the horizontal variables for the time-dependent problem; the spectral stability of the sublaplacian by the complex-valued decaying potential perturbation under an explicit smallness condition on the potential. 
%Moreover, we also obtain several corresponding results for the fractional and bi-sublaplacian cases.

\end{abstract}

%\tableofcontents

\section{Introduction}\label{sec:intro}
Many physical systems evolve within the laws of constrained dynamics, operating within specific geometric settings, where the movement is restricted to predetermined directions within the tangent space. This fundamental premise has spurred the advancement of the sub-Riemannian geometry, which starts with the famous paper by Cartan \cite{C}. A fundamental aspect of sub-Riemannian structures lies in their ability to regain the ``missing directions'' within the tangent space by iteratively computing commutators of vector fields that characterize the partial differential operator under consideration.  Among the large amount of physical applications, we mention Quantum Mechanics, non-Holonomic Mechanics, the Physiology of Neurovision, etc. 
The typical framework for modeling such systems involves the so called non-abelian Carnot groups (or {\it stratified Groups}), which are non-abelian Lie groups $\mathbb G$, (i.e. groups which are also Riemannian manifolds, with a non-commutative group law) with a Lie algebra (i.e. the tangent space at the identity) which exhibits a layered structure, mirroring the underlying physical context.
The toy model is given by the well known $(2d+1)$-dimensional {\it Heisenberg Group} $\mathbb H^d$, introduced by Weil to describe his theoretical approach to Quantum Mechanics based on Group Representation Theory. Naturally connected to the sub-Riemannian structure of $\mathbb H^d$, the {\it sublaplacian} $\mathcal L$ is the hypoelliptic operator defined as the sum of the squares of suitable horizontal vector fields which generate, together with their commutators, the Lie Algebra (see the definition \eqref{eq:sublaplacian} below). 

The purpose of the present manuscript is to describe how the sub-Riemannian (in particular, the anisotropic) structure of $\mathbb H^d$ affects to the global properties of solutions to Schr\"odinger equations associated with $\mathcal L$, as well as its pure fractional power $\mathcal L^s$ and conformally invariant fractional power $\L_s$. We mainly focus on the so-called {\it uniform resolvent estimates} for the stationary problem and {\it Kato-type local smoothing effects} for the time-dependent problem. 

%\subsection{Uniform resolvent estimates and Kato-smoothing effects} 
Let us first introduce the topic in the standard Euclidean setting. 
In the well celebrated paper \cite{KY}, Kato and Yajima proved the following resolvent estimate
\begin{equation}\label{eq:ky1}
\left\|\langle x\rangle^{-\frac12-\ep}\<D\>^{\frac12}(-\Delta-z)^{-1}\<D\>^{\frac12}\langle x\rangle^{-\frac12-\ep}f\right\|_{L^2(\R^d)}\leq C\|f\|_{L^2(\R^d)},
\end{equation}
for any $z\in\mathbb C\setminus[0,\infty)$, $\ep>0$, $d\ge3$, and
for some constant $C>0$, independent on $z$ and $f$. Here $\langle x\rangle:=\sqrt{1+|x|^2}$,  $|D|:=\mathcal F^{-1} |\xi|\mathcal F $ and $\Delta:=-\F^{-1}|\xi|^2\F$ is the Laplacian on $\R^d$. Using the Kato--Yajima terminology, the operator $\langle x\rangle^{-\frac12-\epsilon}\<D\>^{\frac12}$ is $-\Delta$-{\it supersmooth}, and this fact standardly implies by Kato-smoothing (see \cite{Kato}) that the Schr\"odinger evolution flow satisfies the following local smoothing effects
\begin{equation}\label{eq:ky2}
\left\|\langle x\rangle^{-\frac12-\ep}\<D\>^{\frac12}e^{i\tau\Delta}f\right\|_{L^2(\R^{d+1})}\leq C\|f\|_{L^2(\R^d)},
\end{equation}
with the same conditions on the parameters as above. Estimates \eqref{eq:ky2} with $\<D\>^\frac12$ replaced by $|D|^\frac12$ were proved by Ben-Artzi and Klainerman in \cite{BK} by an equivalent method based on the functional calculus representation. As for the case of homogeneous weights, the following family of estimates hold: 
\begin{equation}\label{eq:ky3}
\left\||x|^{-\alpha+\beta}|D|^{\beta}(-\Delta-z)^{-1}|D|^{\beta}|x|^{-\alpha+\beta}f\right\|_{L^2(\R^d)}\leq C\|f\|_{L^2(\R^d)}
\end{equation}
provided $\alpha-\frac d2<\beta<\alpha-\frac12$ (see \cite{KY} for $0\le \beta<\alpha-\frac12$ and \cite{Sugimoto} for $\alpha-\frac d2<\beta<0$). As above, by Kato-smoothing one has
\begin{equation}\label{eq:ky4}
\left\||x|^{-\alpha+\beta}|D|^{\beta}e^{i\tau\Delta}f\right\|_{L^2(\R^{d+1})}\leq C\|f\|_{L^2(\R^d)}.
\end{equation}
A particularly relevant  case is when $\alpha=1,\beta=0$ in \eqref{eq:ky3}, namely
\begin{equation}\label{eq:ky5}
\left\||x|^{-1}(-\Delta-z)^{-1}|x|^{-1}f\right\|_{L^2(\R^d)}\leq C\|f\|_{L^2(\R^d)},
\end{equation}
which at the zero-energy $z=0$ is, at least formally, related with the weighted Rellich inequality
$$
\left\||x|^{-1}f\right\|_{L^2(\R^d)}\leq C\||x|\Delta f\|_{L^2(\R^d)}.
$$
We also mention the paper by Simon \cite{Simon}  in which a general method to produce the sharp constants in \eqref{eq:ky2} is provided. To illustrate the motivation, we also recall the $L^p$-$L^q$ estimates
\begin{align}
\label{eq:ky6}
\|e^{i\tau\Delta}f\|_{L^q(\R^d)}\lesssim |\tau|^{-d(1/2-1/q)}\|f\|_{L^p(\R^d)},\quad \tau\neq0,\ 1\le p\le 2\le q\le\infty,\ 1/p+1/q=1,
\end{align}
and the Strichartz estimates (see \cite{Strichartz, GV, Yajima, KT})
\begin{align}
\label{eq:ky7}
\|e^{i\tau\Delta}f\|_{L^p(\R;L^q(\R^d))}\lesssim \|f\|_{L^2(\R^d)},\quad p\ge2,\ \frac 2p+\frac dq=\frac d2,\ (d,p,q)\neq(2,2,\infty). 
\end{align}
In the last decades, these estimates \eqref{eq:ky1}--\eqref{eq:ky7} turned out to be fundamental to understand the associated time-dispersive flows (see e.g. \cite{Kato, RS}) as well as the properties of the discrete spectrum, even in the case in which the operator of interest is not self-adjoint (see e.g. \cite{BPST1, BPST2, EGS, DF, MY} and references therein). %\subsection{Aim  of the paper} 

We now move to the Heisenberg setting. In contrast with the Euclidean case, a peculiar fact emerges in the study of the Schr\"odinger evolution equation associated to $\mathcal L$, which fails to be dispersive, since it possesses (uncountably many) soliton-like solutions, as it has been observed by Bahouri, G\'erard and Xu in \cite{BGX}. Precisely, for any $g\in C_0^\infty(\R)$ supported in $(0,\infty)$, we set 
$$
f_*(z,t):=\int_\R e^{it\lambda}e^{-\lambda|z|^2}g(\lambda)d\lambda,\quad (z,t)\in \mathbb H^d,
$$
where $z$ and $t$ denote the horizontal and vertical variables in $\mathbb H^d$, respectively (see the beginning of Section \ref{main_result} for basic notions of $\mathbb H^d$). Then the solution $e^{-i\tau \mathcal L}f_*$ to the Schr\"odinger equation $(i\partial_\tau-\L)e^{-i\tau \mathcal L}f_*=0$ with the initial datum $f_*$ is given by the translation in $t$-direction with speed $4d$, namely
$$
e^{-i\tau \mathcal L}f_*(z,t)=f_*(z,t-4d\tau ). 
$$
In particular, one has $$\|e^{-i\tau \mathcal L}f_*\|_{X}=\|f_*\|_{X}$$ for any $\tau\in \R$ and any function space $X$ on $\mathbb H^d$ with norm $\|\cdot\|_X$ invariant under the translation in the $t$-variable. As a consequence, any (global-in-time) translation-invariant dispersive estimate, such as the Strichartz or $L^p$-$L^q$ estimates, {\it cannot hold} for $e^{-i\tau \mathcal L}$. Recently, Bahouri, Barilari and Gallagher in \cite{BBG} %, and Bahouri and Gallagher in \cite{BG} 
however showed that, despite the lack of global dispersion, one can recover a weaker version of Strichartz estimates than the Euclidean case by restricting initial data to {\it cylindrical}\,\footnote{$f(z,t)$ is said to be  \emph{cylindrical} if $f(z,t)=\tilde f(|z|,t)$ with some $\tilde f \colon [0,\infty)\times \R\to \C$, and be \emph{radial} if  $f(z,t)=\tilde f(|(z,t)|_{\mathbb H})$ with some $\tilde f \colon [0,\infty) \to \C$, where $|(z,t)|_{\mathbb H}=(|z|^4+t^2)^{1/4}$ denotes the Koranyi norm. Note that the literature is not unanimous in the use of this nomenclature: the functions that here, sticking to the terminology of e.g. \cite{GL} and of our previous work \cite{FMRS}, we call cylindrical, are sometimes denominated radial.}  functions and by suitably averaging in time $\tau$ and space with respect to the horizontal variable $z$ and then taking the supremum in the vertical variable $t$. Several space-time weighted $L^2$-estimates {\it with loss of derivatives} were also proved by M\u{a}ntoiu \cite{Mantoiu} (see Section \ref{subsection_known} for more details). 

These results motivated us to start a parallel project in \cite{FMRS} in which we investigate the validity of uniform resolvent estimates for {\it radial} solutions via the multiplier method (see \cite{BVZ, CFK, CK, FKV, FKV2} and the references therein), and, by the usual theory of $H$-smoothness by Kato in \cite{Kato}, the consequent Kato-smoothing effects. This paper is a continuation of \cite{FMRS} to further investigate uniform resolvent estimates on $\mathbb H^d$ {\it without any symmetry} or {\it derivative loss}. 
%In the recent years, we have developed a multipliers' method in order to obtain estimates of the type \eqref{eq:ky5} for perturbed Hamiltonian (see \cite{BVZ, CFK, CK, FKV, FKV2} and the references therein). In a first attempt to generalize the method to the setting of the Heisenberg Group \cite{FMRS}, we needed to take into account the existence of the Bahouri-G\'erard-Xu soliton, which is responsible of a lack of positivity of some commutators, and obliges us to restrict the study to radial solutions (see Remark \ref{rmk:cfFMRS23} below for a comparison with the main results of the present manuscript). 
%In the present manuscript we suitably involve the basic tools from Mourre's theory in order to obtain some analogous to the above smoothing estimates in the Heisenberg setting, both for the sublaplacian and for its powers. 
Our main results are Theorem~\ref{theorem_resolvent_general} and  Theorem~\ref{theorem_resolvent_radial}. In Theorem~\ref{theorem_resolvent_general}, we provide uniform resolvent estimates for $\L$, as well as the fractional sublaplacians $\L^s$ and $\L_s$,  between $L^2$ spaces weighted with suitable operators related with Hardy's type inequalities on $\mathbb H^d$ (and thus the geometry of $\mathbb H^d$). In Theorem~\ref{theorem_resolvent_radial}, under suitable anisotropic symmetric assumption, we see how some of these weighted operators change and particularly recover similar uniform resolvent estimates as in the Euclidean case.   As already explained in the Euclidean case, our uniform resolvent estimates can be directly applied to obtain smoothing effects for the Schr\"{o}dinger equation and to the spectral theory of perturbed operators---these results are reported in Section~\ref{sec:applications}.

%There are interesting points to emphasize compared with previous results, specially in the second order case $s=1$.  \begin{enumerate}\item We obtain uniform estimates \textit{without any symmetry or derivative loss}, see  (i) and (ii) in Theorem~\ref{theorem_resolvent_general};\item Compared with M\u{a}ntoiu \cite{Mantoiu}, full derivative losses are not necessary and spherical ones are enough, see  (iii) and (iv) in Theorem \ref{theorem_resolvent_general}. In particular, we recover a similar local smoothing effects to the Euclidean case under cylindrical symmetry (Theorem \ref{theorem_resolvent_radial}). \end{enumerate}

\section{Main results}
\label{main_result}

We first recall some definitions and basic notions of the functional framework on the Heisenberg group (see e.g. the standard textbook \cite{Folland2} for a detailed description). 

Let $d\in \N$ and $\mathbb H^d=\R^{2d+1}$ be the $d$-dimensional Heisenberg group endowed with the group law
$$
(z,t)\cdot(z',t')=(x+x',y+y',t+t'+2x'\cdot y-2x\cdot y'),
$$
where $z=(x,y)=(x_1,\dots,x_d,y_1,\dots,y_d), z'=(x',y') \in \R^d\times \R^d$, $t,t'\in \R$ and $x\cdot y=x_1y_1+\cdots+x_dy_d$. 
Let us also equip $\mathbb H^d$ 
with the metric structure induced by the Koranyi norm 
\begin{align}
\label{Koranyi}
|(z,t)|_{\mathbb H}:=(|z|^4+t^2)^{1/4}
\end{align}
 and
with the Haar measure $dzdt$, and we define $L^2(\mathbb H^d)$ with the usual inner product and norm
$$
\<f,g\>:=\int_{\mathbb H^d}f(z,t)\overline{g(z,t)}dzdt,\quad \|f\|_{L^2(\mathbb H^d)}:=\left(\int_{\mathbb H^d}|f(z,t)|^2dzdt\right)^{1/2}. 
$$
For $j=1,...,d$, we introduce $X_j,Y_j$  and $T$ the left invariant vector fields on $\mathbb H^d$ defined by
$$
X_j=\partial_{x_j}+2y_j\partial_t,\quad Y_j=\partial_{y_j}-2x_j\partial_t,\quad T=\partial_t,
$$
and the span $\mathcal D:=\text{span}\{X_j,Y_j\}_{j=1,\dots,d}$ which provides a sub-Riemannian structure on $\Heis^d$. 
These vector fields satisfy the commutation relations
\begin{equation}\label{eq:commutators}
	[X_j,X_k] = [Y_j,Y_k] = 0, 
	\qquad
	[X_j,Y_k] = - 4 \delta_{jk} T,
\end{equation}
with $j,k = 1, \dots, d$ and $\delta_{jk}$ being the Kronecker's delta.

Setting the horizontal gradient $$\nabla_{\mathbb H}=(X,Y)=(X_1,...,X_d,Y_1,...,Y_d),
$$ we define the sublaplacian $\L$ on $\mathbb H^d$ by
\begin{equation}\label{eq:sublaplacian}
\L = -\Delta_{\mathbb H} = -\nabla_{\mathbb H}\cdot \nabla_{\mathbb H} = -\sum_{j=1}^d (X_j^2+Y_j^2).
\end{equation}
The operators $\nabla_{\mathbb H}$ and $\Delta_{\mathbb H}$ can be also written in the form 
\begin{align}
\label{sublaplacian_2}
\nabla_{\mathbb H}=\partial_z+2z^\perp T,\quad \Delta_{\mathbb H}=\partial_z^2+4NT+4|z|^2T^2,
\end{align}
where $z^\perp:=(y,-x)$ and
	$$
	N:=z^\perp\cdot\partial_z = \partial_z\cdot z^\perp =\sum_{j=1}^d(y_j\partial_{x_j}-x_j\partial_{y_j}).
	$$
It is known that $\L$ is a non-negative hypoelliptic differential operator, which is essentially self-adjoint on the Schwartz class $\S(\mathbb H^d)=\S(\R^{2d+1})$ and thus admits a unique self-adjoint extension on $L^2(\mathbb H^d)$ with the maximal domain 
$$
D(\L):=\{f\in L^2(\mathbb H^d)\ |\ \L f\in L^2(\mathbb H^d)\}.
$$
With a slight abuse of notation, we still use $\L$ for denoting such a unique extension. 
It is also known that $\L$ has the purely absolutely continuous spectrum $[0,\infty)$. Precisely, $
\sigma(\L)=\sigma_{\mathrm{ac}}(\L)=[0,\infty)$ and $\sigma_{\mathrm{sc}}(\L)=\sigma_{\mathrm{pp}}(\L)=\varnothing$, 
where $\sigma(\L)$, $\sigma_{\mathrm{ac}}(\L)$, $\sigma_{\mathrm{sc}}(\L)$ and $\sigma_{\mathrm{pp}}(\L)$ denote the spectrum, the absolutely continuous, singular and pure point spectrum of $\L$, respectively.  In particular, $\L$ has no eigenvalues: $\sigma_{\mathrm{p}}(\L)=\varnothing$, with $\sigma_{\mathrm{p}}(\L)$ denoting the point spectrum. We refer to \cite[Section 4]{DMW} (see also \cite{J}) for more details on the basic spectral properties of $\L$. We recall also that $\L$ is \emph{homogeneous of degree 2} with respect to the natural dilation $\delta_\lambda$ on the Heisenberg group defined by
\begin{equation}\label{scaling}
	\delta_\lambda(z,t) = (\lambda z, \lambda^2 t), \quad \lambda>0,
\end{equation}
in the sense that $\L(f\circ \delta_\lambda) = \lambda^2 (\L f) \circ \delta_\lambda$. See the beginning of Section \ref{section_abstract} for the dilation group on $\mathbb H^d$. We emphasize that this distinctive dilation group and the Koranyi norm \eqref{Koranyi} particularly represent the anisotropic structure of $\mathbb H^d$. 

In addition to the sublaplacian $\mathcal L$, we also consider in this work the pure fractional sublaplacian $\mathcal L^s:=(\mathcal L)^s$ for $s>0$ and the conformal fractional sublaplacian $\mathcal L_s$ for $0<s<d+1$. We refer to Appendix\til\ref{sec:Ls} below for their definitions, as well as backgrounds.  

%%%

To state the results, we further introduce four weight functions $w_i$, $i=1,\dots,4$, below
\begin{align*}
w_1(z,t)&=\<z\>^{-1}\<|(z,t)|_{\mathbb H}\>^{-2}|z|,& w_2(z,t)&=(1+|z|^2+|t|^2)^{-1/2},\\ w_3(z,t)&=\<|(z,t)|_{\mathbb H}\>^{-2}|z|,& w_4(z,t)&=|(z,t)|_{\mathbb H}^{-2}|z|,%\\ w_5(z,t)&=\<|(z,t)|_{\mathbb H}\>^{-1/2},\quad w_6(z,t)=|(z,t)|_{\mathbb H}^{-1/2},
\end{align*}
where $\<\cdot\>=\sqrt{1+|\cdot|^2}$. We record a few basic properties of these weights: 
\begin{itemize}
\item $w_1\le w_3\le w_4$ on $\mathbb H^d$;
\item $|w_j|\le 1$ on $\mathbb H^d$ for $j=1,2,3$, while $w_4$ blows up at the origin $(z,t)=(0,0)$;
\item $w_1,w_3=0$ on the $t$-axis $\{(0,t)\in \mathbb H^d\ |\ t\in \R\}$ and $w_4=0$ on $\{(0,t)\in \mathbb H^d\ |\ t\neq0\}$. 
\item $w_4$ is homogeneous of degree $-1$ with respect to the scaling \eqref{scaling}. %: $w_4(\lambda z,\lambda^2 t) = \lambda^{-1} w_4(z,t)$.
\end{itemize}

We can now state our main results. In the following Theorems \ref{theorem_resolvent_general}, \ref{theorem_resolvent_radial}, \ref{theorem_smoothing_general} and \ref{theorem_smoothing_radial},  by $\H$ we will denote either $\L^s$ with $s>0$, or $\L_s$ with $0< s<d+1$.

%theorem
\begin{theorem}[Resolvent estimates -- general case]
\label{theorem_resolvent_general}
Let $s>0$ and $\H=\L^s$, or let $0< s<d+1$ and $\H=\L_s$.
Suppose also that $d,s,\mu$ and $G$ satisfy one of the following four conditions: 
\begin{enumerate}[label=(\roman*)]
\item $d\ge1$, $1/2<s\le 1$ and $G=w_1^s\<\L\>^{(s-1)/2} $;
\item $d\ge2$, $1/2<s\le 1$ and $G=w_2^s\<\L\>^{(s-1)/2} $;
\item $d\ge2$, $s>0$, $1/2<\mu\le1$ and $G=\<N\>^{-\mu} w_3^\mu\L^{s/2}\<\L\>^{-1/4}$;
\item $d\ge2$, $s>0$, $1/2<\mu\le1$ and $G=\<N\>^{-\mu} w_4^\mu \L^{(s-\mu)/2}$.
%\item[(A5)] $d\ge3$, $s>0$, $1<\mu\le2$ and $G=w_5^\mu \L^{s/2}\<\L\>^{-1/2}$;\item[(A6)] $d\ge3$, $s>0$, $1<\mu\le2$ and $G=w_6^\mu \L^{(s-\mu)/2}$. \item[(A7)] $d\ge3$, $0<s<d+1$, $2\le \mu<d+1$ and $G=w_6^{-\mu}\L^{(s-\mu)/2}$
\end{enumerate}
Then the uniform resolvent estimate
\begin{align}
\label{theorem_resolvent_general_1}
\sup_{\sigma\in \C\setminus\R}\|G(\H-\sigma)^{-1}G^*f\|_{L^2(\mathbb H^d)}
\lesssim
 \|f\|_{L^2(\mathbb H^d)},\quad f\in \S(\mathbb H^d),
\end{align}
holds with the implicit constant being independent of $f$. 
In particular, $G(\H-\sigma)^{-1}G^*$ extends to a bounded operator on $L^2(\mathbb H^d)$ with a uniform operator norm bound with respect to ${\sigma\in \C\setminus\R}$. 
\end{theorem}

More generally, the above results hold for any non-negative self-adjoint operator $\H$ on $L^2(\mathbb H^d)$ with domain $H^{2s}(\mathbb H^d)$ (see Subsection \ref{sub:notation} for the definition of Sobolev space), satisfying the following three conditions:

\begin{enumerate}[label=(H\arabic*)]
	\item\label{H1}
	$\H$ is homogeneous of degree $2s$, in the sense that 
	\begin{equation*}
		\H(f \circ \delta_\lambda) = \lambda^{2s} (\H f) \circ \delta_\lambda, \quad f\in H^{2s}(\Heis^d),\quad \lambda>0,
	\end{equation*}
	or equivalently, using the generator of the dilation group introduced at the beginning of Section\til\ref{section_abstract}, that for all $\tau\in \R$
	\begin{align}
		\label{homogeneous}
		e^{-i\tau A}\H e^{i\tau A}=e^{2s\tau}\H\quad \text{on}\quad H^{2s}(\mathbb H^d). 
	\end{align}
	
	\item\label{H2}
	There exist constants $C_1,C_2>0$ such that $C_1\L^s\le \H\le C_2\L^s$ in the form sense, i.e.,
	$$
	C_1\|\L^{s/2}f\|_{L^2(\mathbb H^d)}^2\le \|\H^{1/2}f\|_{L^2(\mathbb H^d)}^2\le C_2\|\L^{s/2}f\|_{L^2(\mathbb H^d)}^2,\quad f\in H^{s}(\mathbb H^d). 
	$$
	
	\item\label{H3}
	$\H$ strongly commutes with $\L$, in the sense that 
	$$
	[(\H-\sigma)^{-1},(\L-w)^{-1}]=0 \quad\text{for any $\sigma,w\in \C\setminus\R$.}
	$$
\end{enumerate}
These hypothesis will be employed in the proof of the abstract Theorem\til\ref{theorem_abstract} in Section \ref{section_abstract} below, and it can be checked (see Appendix \til\ref{sec:Ls}) that $\L^s$ and $\L_s$ satisfy them.

%remark
\begin{remark}
The weight $w_4$ naturally appears in the Hardy type inequality on $\mathbb H^d$ of the form
$$
\|w_4f\|_{L^2(\mathbb H^d)}\le d^{-1}\|\L^{1/2}f\|_{L^2(\mathbb H^d)}%,\quad \|w_6^2f\|_{L^2(\mathbb H^d)}\le\{d(d-2)\}^{-1} \|\L f\|_{L^2(\mathbb H^d)}
$$
(see Lemma \ref{lemma_Hardy} for more details). Note that the operator $w_4\L^{-1/2}$ is invariant under the scaling \eqref{scaling}. In particular, $w_4$ is critical for the estimate \eqref{theorem_resolvent_general_1} (at least near the zero energy $\sigma=0$) in view of the scaling structure, decay rate at infinity and singularity at the origin. On the other hand, the other weights are somewhat artificial and depend on our method. In fact, one can find from the proof of this theorem, precisely Proposition \ref{proposition_replace} and its proof, that the above specific formulas of $w_1,w_2,w_3$ are not necessary, but the following conditions are sufficient: 
\begin{align*}
w_1 &\le1,\quad w_1|z|\le1, \quad w_1|t|\le1,\quad w_1 \le w_4;  \\
w_2 &\le1,\quad  w_2|z|\le1, \quad w_2|t|\le1;\\
w_3 &\le1,\quad  w_3|z|\le1,\quad  w_3|t|\le |z|,\quad w_3 \le w_4,
%\item $w_5\le1$, $w_5|t|\le1$ and $w_5\le w_6$. 
\end{align*}
and that $w_3$ is cylindrical.%, see Remark\til\ref{rem:N}.
	%namely $w_3(z,t)=\widetilde{w_3}(|z|,t)$ for some $\widetilde{w_3}\colon\R^2\to\R$.
\end{remark}

\begin{comment}
	\textcolor{purple}{
		\begin{remark}
			$s\neq1/2$ due to Kato-smoothing.
		\end{remark}
	}
\end{comment}

%remark
\begin{remark}\label{rem:N}
There is a factor $\<N\>^{-\mu}$ in the third and fourth choices of $G$ in Theorem \ref{theorem_resolvent_general}, which describes a kind of spherical derivatives in the $z$-variable. Indeed, $N$ is the form
$
N=\sum_{j=1}^d L_{d+j,j},
$
where $L_{j,k}=z_j\partial_{z_k}-z_k\partial_{z_j}$ is the infinitesimal generator of the pull back $R_{jk}(\theta)^*f(z)=f(R_{jk}(\theta)z)$ by the rotation group
$$
R_{jk}(\theta):\R^{2d}\ni z\mapsto (z_j\cos\theta-z_k\sin\theta)e_j+(z_j\sin\theta+z_k\cos\theta)e_k+\sum_{\ell\neq j,k}z_\ell e_\ell\in \R^{2d}
$$
so that $\Delta_{S^{2d-1}}=\sum_{j<k}L_{j,k}^2$, where $(e_j)_{j=1}^{2d}$ denotes the standard basis of $\R^{2d}$ and $\Delta_{S^{2d-1}}$ denotes the spherical Laplacian on $\R^{2d}$. In particular, we have that $0\le -N^2\le -d\Delta_{S^{2d-1}}$ in the form sense, since
\begin{align*}
	\<-N^2u,u\>_{L^2(\mathbb H^d)}%=\|N u\|^2
%	&=\left\| \sum_{j=1}^d L_{d+j,j}u\right\|^2
	\le d\sum_{j=1}^d\|L_{j,d+j}u\|_{L^2(\mathbb H^d)}^2
	\le d\sum_{1\le j<k\le 2d}\|L_{j,k} u\|_{L^2(\mathbb H^d)}^2
	%=d\sum_{j<k}\<L_{j,k}^2u,u\>
	=-d\<\Delta_{S^{2d-1}}u,u\>_{L^2(\mathbb H^d)}
\end{align*}
and that Theorem \ref{theorem_resolvent_general} still holds with $\<N\>^{-\mu}$ replaced by $\<\Delta_{S^{2d-1}}\>^{-\mu/2}$. 
%\begin{align*}\<Nu,v\>=\sum_{j=1}^d\left(\<z_{d+j}\partial_{z_j}u,v\>-\<z_j\partial_{z_{d+j}}u,v\>\right)=\sum_{j=1}^d\left(u,-z_{d+j}\partial_{z_j}v\>+\<u,z_j\partial_{z_{d+j}}v\>\right)=\<u,-Nv\>\end{align*}
Moreover, $Nf=0$ if $f$ is \emph{polyradial}, namely $f(z,t)=\tilde f(|(x_1,y_1)|,\dots,|(x_d,y_d)|,t)$ with some $\tilde f \colon [0,\infty)^d \times \R\to \C$. A particular case of polyradial functions are the cylindrical or radial functions. When $d=1$, it is easy to check that the cylindrical functions are the only ones in $\Ker N$.
%
%so $Nf=0$ if $f$ is radial with respect to $z$, namely $f(z,t)=\tilde f(|z|,t)$ with some $\tilde f:[0,\infty)\times \R\to \C$. 
\end{remark}

Although it is not clear whether the derivative loss $\<N\>^{-\mu}$ in the above theorem can be removed or not in the general case, this is the case for $f\in \Ker N$ as follows: 

%theorem
\begin{theorem}[Resolvent estimates on $\Ker N$]
\label{theorem_resolvent_radial}
Let $s>0$ and $\H=\L^s$, or $0< s<d+1$ and $\H=\L_s$.
Suppose $d\ge1$, $1/2<\mu\le1$ and $f\in \S(\mathbb H^d)\cap \Ker N$. Then \eqref{theorem_resolvent_general_1} holds for $G=w_3^\mu\L^{s/2}\<\L\>^{-1/4}$ and $G=w_4^\mu \L^{(s-\mu)/2}$. 
\end{theorem}

Again, Theorem \ref{theorem_resolvent_radial} can be generalized to any non-negative self-adjoint operator $\H$ with domain $D(\H)=H^{2s}(\Heis^{d})$ and satisfying assumptions \ref{H1}, \ref{H2} and \ref{H3}, if however in addition $\H$ leaves $\Ker N$ invariant, in the sense that $(\H-\sigma)^{-1}\Ker N\subset \Ker N$ for all $\sigma\in \C\setminus\R$.
%\footnote{The fact $(\H-\sigma)^{-1} \Ker N \subset \Ker N$ is used to ensure $G(\H-\sigma)^{-1}G^*f=P_NG(\H-\sigma)^{-1}G^*f$ for $f\in \Ker N$}

\begin{remark}
	Theorem \ref{theorem_resolvent_radial} is not a direct corollary of Theorem \ref{theorem_resolvent_general} since it holds for all $d\ge1$, while the results for $d\ge2$ follows from Theorem \ref{theorem_resolvent_general} (with a possible change of the implicit constants in \eqref{theorem_resolvent_general_1}). 
The difference in the dimensional hypothesis arise from the employment, in the proof of Theorem\til\ref{theorem_resolvent_general}, of the D'Ambrosio-type Hardy inequality \eqref{lemma_Hardy_2}.
\end{remark}

%\begin{remark}If we consider the case of the bi-sublaplacian, we can also prove a better result than Theorem \ref{theorem_resolvent_general} without derivative loss $\langle N\rangle^{-\mu}$, see Appendix \ref{bisublaplacian}.\end{remark}

\begin{remark}\label{rmk:cfFMRS23}
Theorem \ref{theorem_resolvent_general} with $s=1$ and $G=w_1$ or $w_2$,  and 
Theorem \ref{theorem_resolvent_radial} with $s=\mu=1$ and $G=w_4$ should be compared with \cite[Theorem 1]{FMRS} by the authors. Indeed, considering a \emph{radial} weak solution $u \in H^1(\mathbb{H}^d)$ of the Helmholtz equation $\L u - (\sigma_1+i\sigma_2)u = f$, and setting 
	\begin{equation*}
		u^- (z,t)
		:=
		e^{- i \psi(z,t)} u(z,t),
		\quad
		\psi(z,t)
		:=
		\sqrt{\frac{d}{2}} \frac{\Gamma(\frac{d}{2})}{\Gamma(\frac{d+1}{2})} \sgn(\sigma_2) \sqrt{|\sigma_1|} 
		\Kor,
	\end{equation*}
where $\Gamma(\cdot)$ is the Gamma function, 
	in \cite{FMRS} we have proved via the multipliers method that
	\begin{equation*}
		\begin{aligned}
		\norm{\grad u}_{L^2({\mathbb H}^d)} &\lesssim  \norm{ w_4^{-1} f}_{L^2({\mathbb H}^d)}
		&& \text{if $|\sigma_2|>\delta\sigma_1$,}
		\\
		\norm{\grad u^-}_{L^2({\mathbb H}^d)} &\lesssim  \norm{ w_4^{-1} f}_{L^2({\mathbb H}^d)}
		&& \text{if $|\sigma_2|\le\delta\sigma_1$,}
		\end{aligned}
	\end{equation*}
	for any $\delta>0$.
	These are Barcel\'{o}--Vega--Zubeldia-type estimates (see \cite{BVZ}), stronger than the Kato--Yajima-type estimate (see \cite{KY})
	\begin{align*}
		\sup_{\sigma\in \C\setminus\R}\| w_4 (\L-\sigma)^{-1} w_4 f\|_{L^2(\mathbb H^d)}\lesssim \|f\|_{L^2(\mathbb H^d)},\quad f\in \S(\mathbb H^d),
	\end{align*}
	which corresponds to \eqref{theorem_resolvent_general_1}, and follows from the first ones employing Hardy's inequality \eqref{lemma_Hardy_1}. However, no symmetric condition is needed in Theorem \ref{theorem_resolvent_general} and the symmetric condition in Theorem\til\ref{theorem_resolvent_radial} is less restrictive compared to the one in \cite{FMRS}: if $u$ is radial, then $f\in\Ker N$. %Moreover, $u$ is not necessarily radial even if $f$ is radial. 
\end{remark}

	%remark
	\begin{remark}
		\label{remark_operator_norm}
		Although it is far from being optimal, our proof gives some explicit upper bounds of the implicit constants in \eqref{theorem_resolvent_general_1} for all the above choices of $G$ in Theorems \ref{theorem_resolvent_general} and \ref{theorem_resolvent_radial}. 
		For instance, 
		\begin{align*}
			&\|w_1(\L-\sigma)^{-1}w_1\|\le 2(3e^{1/4}-2)^2 (5+3d^{-1})^2,\\
			&\|w_2(\L-\sigma)^{-1}w_2\|\le 2(3e^{1/4}-2)^2 (5+4(d-1)^{-1})^2,\\
			&\|\<N\>^{-1} w_3 \L^{1/2}\<\L\>^{-1/4}(\L-\sigma)^{-1} \<\L\>^{-1/4} \L^{1/2} w_3 \<N\>^{-1}\|
			\le 2(3e^{1/4}-2)^2 (4+d^{-1}+(d-1)^{-1})^2,\\
			&\| \<N\>^{-1} w_4 (\L-\sigma)^{-1} w_4 \<N\>^{-1} \| \le 2(3e^{1/4}-2)^2 (3+(d-1)^{-1}).
		\end{align*}
		We refer to Theorem \ref{theorem_abstract}, Remark \ref{remark_kappa}, Proposition \ref{proposition_replace}, and the proof of Theorem \ref{theorem_resolvent_general} below for more details. 
	\end{remark}

\subsection{Smoothing effect and spectral stability}\label{sec:applications}

In this subsection we prove two typical applications of the above uniform resolvent estimates. 

\subsubsection{Smoothing estimates}
By virtue of the theory of smooth perturbations, the above theorems immediately imply associated smoothing effects for the solution to the Schr\"odinger-type equation
\begin{align}
\label{SE}
(i\partial_\tau-\H)u(\tau,z,t)=F(\tau,z,t),\quad u(0,z,t)=u_0(z,t),\quad \tau\in \R,\ (z,t)\in \mathbb H^d,
\end{align}
with given data $u_0$ and $F$, where the solution is given by the Duhamel formula: 
$$
u(z,t)=e^{-i\tau\H}u_0(z,t)-i\int_0^\tau e^{-i(\tau-r)\H}F(r,z,t)dr. 
$$
%Note that $t$ denotes the vertical variable in $\mathbb H^d$ and $\tau$ denotes the time variable in this paper. 

%theorem
\begin{theorem}[Smoothing estimates for general data]
\label{theorem_smoothing_general}
Under the same conditions in Theorem \ref{theorem_resolvent_general}, one has for $u_0\in \S(\mathbb H^d)$ and $F\in \S(\R\times \mathbb H^d)$, 
\begin{align}
\label{theorem_smoothing_general_1}
\|Ge^{-i\tau \H}u_0\|_{L^2(\R\times\mathbb H^d)}&\lesssim \|u_0\|_{L^2(\mathbb H^d)},\\
\label{theorem_smoothing_general_2}
\left\|G\int_0^\tau e^{-i(\tau-r)\H}G^*F(r)dr\right\|_{L^2(\R\times\mathbb H^d)}&\lesssim \|F\|_{L^2(\R\times\mathbb H^d)}.
\end{align}
\end{theorem}
Indeed, for any densely defined closed operator $G$ on $L^2(\mathbb H^d)$ with domain inclusion $D(\H)\subset D(G)$, it is known (see \cite{Kato,DAncona}) that \eqref{theorem_smoothing_general_2} is equivalent to \eqref{theorem_resolvent_general_1}, while \eqref{theorem_smoothing_general_1} is equivalent to 
$$
\sup_{\sigma\in \C\setminus\R}\|G\{(\H-\sigma)^{-1}-(\H-\overline{\sigma})^{-1}\}G^*f\|_{L^2(\mathbb H^d)}\lesssim \|f\|_{L^2(\mathbb H^d)}, 
$$
which clearly follows from \eqref{theorem_resolvent_general_1}. 
Similarly, we have also the following

%theorem
\begin{theorem}[Smoothing estimates for data in $\Ker N$]
\label{theorem_smoothing_radial}
Let $d\ge1$, $s>0$, $1/2<\mu\le1$, $u_0\in \mathcal S(\mathbb H^d)\cap \Ker N$ and $F\in \mathcal S(\R\times \mathbb H^d)\cap L^1_{\loc}(\R;\Ker N)$. Then \eqref{theorem_smoothing_general_1} and \eqref{theorem_smoothing_general_2} hold for $G=w_3^\mu\L^{s/2}\<\L\>^{-1/4}$ and $G=w_4^\mu \L^{(s-\mu)/2}$. 
\end{theorem}

\subsubsection{Spectral stability}
In the last decades, Spectral analysis of non-self-adjoint operators has been extensively studied (see e.g. \cite{CFK,CK,FKV,FKV2,HK} and references therein). A natural and typical question in this context is the spectral stability, namely under what conditions on the perturbations are the spectrum of perturbed operators  the same as that of the unperturbed one? It is worth mentioning that, in contrast with the self-adjoint case, many classical tools such as the comparison principle and the spectral decomposition theorem are not available in the non-self-adjoint setting in general. Instead, suitable resolvent estimates (combined with the Birman--Schwinger principle) have been shown to be very useful to prove the spectral stability or certain spectral radius bounds for eigenvalues (see \cite{Frank}). 

Our theorem can be also used to show the spectral stability for the Schr\"odinger operator $\L+V$ with a complex-valued potential $V(z,t)$ as follows. Recall that $\L$ is the positive self-adjoint operator on $L^2(\mathbb{H}^d)$ associated to the quadratic~form
$$
q_0[\psi]:=\int_{\mathbb{H}^d}|\nabla_{\!\mathbb{H}}\psi|^2dzdt,
\qquad
D(q_0)=H^1(\mathbb{H}^d),
$$
with purely continuous spectrum $[0,\infty)$. %Above and in what follows, $H^1(\mathbb{H}^d)$ is the Sobolev space defined in Subsection \ref{sub:notation}. 
Let $V:\mathbb{H}^d\to\C$ be a measurable function subordinated to $\L$, with bound less than one, viz. there exists $0\le a<1$ such that
\begin{equation}\label{eq:subordination}
%	\text{there exists}\ a<1 \ \text{such that}\ 
	\int_{\mathbb{H}^d}|V||\psi|^2dzdt\leq a\int_{\mathbb{H}^d}|\nabla_{\!\mathbb{H}}\psi|^2dzdt,
\end{equation}
for all $\psi\in H^1(\mathbb{H}^d)$. Then the quadratic form
$$
q_V[\psi]:=\int_{\mathbb{H}^d}V|\psi|^2dzdt,
\quad
D(q_V):=\Big\{\psi\in L^2(\mathbb{H}^d)\ |\ \int_{\mathbb{H}^d}|V||\psi|^2dzdt<\infty\Big\}
$$
is relatively bounded with respect to $q_0$, with bound less than 1, and we can define an \mbox{$m$-sectorial} differential operator $\L+V$ associated to the closed quadratic form $q:=q_0+q_V$. %, since this is a closed form (see \cite[Chapter VI, Theorem 2.1]{K}). 

%theorem
\begin{theorem}[Spectral stability]
\label{theorem_spectral}
Suppose that $V:\mathbb H^d\to \C$ is a measurable function satisfying \eqref{eq:subordination}, and that $V$ and $d$ satisfy one of the following conditions: 
\begin{enumerate}[label=(\roman*)]
	\item $d\ge1$ and $|V|\le C_1 w_1^{2}$ almost everywhere, with some 
	$$
	C_1 < \frac12 (3e^{1/4}-2)^{-2} (5+3d^{-1})^{-2} ,
	$$
	\item $d\ge2$ and $|V|\le C_2 w_2^{2}$ almost everywhere, with some 
	$$
	C_2 < \frac12 (3e^{1/4}-2)^{-2} (5+4(d-1)^{-1})^{-2}.
	$$
\end{enumerate}
Then $\sigma(\L+V)=[0,\infty)$ and $\sigma_{\mathrm{p}}(\L+V)=\varnothing$. 
\end{theorem}

%proof  
\begin{proof}
We write $V=B^*A$ with $A=\sqrt{|V|}$ and $B^*=\sgn(V)\sqrt{|V|}$, and consider the Birman--Schwinger operator $A(\L-\sigma)^{-1}B^*$ with $\sigma\in \C\setminus\R$. By the Birman--Schwinger principle (see \cite[Theorem 3]{HK} and also \cite{Kato,KK}), the claim of the theorem follows if $A\<\L\>^{-1/2}$ and $B\<\L\>^{-1/2}$ are bounded on $L^2(\mathbb H^d)$ and 
\begin{align}
\label{theorem_spectral_proof_1}
\sup_{\sigma\in \C\setminus[0,\infty)}\|A(\L-\sigma)^{-1}B^*\|_{\mathbb B(L^2(\mathbb H^d))}<1. 
\end{align}
The former condition is trivial since $w_1\le1$ and $w_2\le1$ on $\mathbb H^d$. The assumption in this theorem, Theorem \ref{theorem_resolvent_general} and Remark \ref{remark_operator_norm} imply $\sup_{\sigma\in \C\setminus\R}\|A(\L-\sigma)^{-1}B^*\|_{\mathbb B(L^2(\mathbb H^d))}\le c$ with some $c<1$. Since the map $\C\setminus[0,\infty) \ni \sigma \mapsto A(\L-\sigma)^{-1}B^*\in \mathbb B(L^2(\mathbb H^d))$ is continuous, we thus obtain \eqref{theorem_spectral_proof_1}. 
\end{proof}

\begin{remark}
	Assuming the subordination condition \eqref{eq:subordination} in Theorem\til\ref{theorem_spectral} is actually redundant. Indeed, condition (i) with the Hardy-type inequality \eqref{lemma_Hardy_1} imply \eqref{eq:subordination} with
	$a=d^{-2} C_1$, 
	whereas 
	condition (ii) in the theorem together with the Hardy-type inequality \eqref{lemma_Hardy_2} imply \eqref{eq:subordination} with
	$a=(d-1)^{-2}C_2$.
	It is easy to see that in both cases one has $a<1$.
\end{remark}

\begin{remark}
We cannot consider an analogous result to Theorem \ref{theorem_spectral} for the fractional versions $\mathcal{L}_s$ and $\mathcal{L}^s$.  Indeed, observe that when $s=1$ in  the two first conditions of Theorem  \ref{theorem_resolvent_general}, the weight $G$ boils down to either $w_1$ (in the case $d\ge 1$) or $w_2$ (in the case $d\ge2$). Nevertheless, if $1/2<s<1$, the weight $G$ is an operator and it is not possible to impose a suitable assumption for $V$ on Theorem \ref{theorem_spectral}.
\end{remark}

\subsection{Known results}
\label{subsection_known}
We here recall some known results and compare our theorems with them. In \cite{BBG}, Bahouri, Barilari and Gallagher proved following Strichartz-type estimates: 
$$
\|e^{-i\tau \mathcal L}f\|_{L^\infty_tL^p_\tau L^q_z}\lesssim \|f\|_{L^2}
$$
for any $2\le p,q\le \infty$, $2/p+2d/q=d+1$ and cylindrical $f\in L^2(\mathbb H^d)$. This estimate with the choice  $p=q=2$ and H\"older's inequality particularly imply
$$
\|w(t) e^{-i\tau \mathcal L}f\|_{L^2(\R\times \mathbb H^d)}\lesssim \|w\|_{L^2(\R)}\|f\|_{L^2(\mathbb H^d)}
$$
for any cylindrical $f\in L^2(\mathbb H^d)$ and $w\in L^2(\R)$. In \cite{Mantoiu}, M\u{a}ntoiu proved several space-time estimates with loss of derivatives for the sublaplacians on stratified Lie groups, which, in the setting of Heisenberg groups, particularly implies, for $1/2<s\le1$, 
$$
\|\<|(z,t)|_{\mathbb H}\>^{-s}|(z,t)|_{\mathbb H}^{-s}\<\mathcal L\>^{-s/2}\mathcal L^{(1-s)/2}e^{-i\tau \mathcal L}f\|_{L^2(\R\times \mathbb H^d)}\lesssim\|f\|_{L^2(\mathbb H^d)}. 
$$
Note that the term $\<\mathcal L\>^{-s/2}\mathcal L^{(1-s)/2}$ represents the derivative losses if $s>1/2$. 
Compared with these two results, the interesting point of our theorems is that, on one hand, we obtain uniform estimates \textit{without any symmetry or derivative loss} ((i) and (ii) in Theorems~\ref{theorem_resolvent_general} and \ref{theorem_smoothing_general}), and, on the other hand, full derivative losses with respect to $\mathcal L$ are not necessary and spherical ones are enough ((iii) and (iv) in Theorems \ref{theorem_resolvent_general} and \ref{theorem_smoothing_general}). In particular, we recover a similar uniform resolvent estimates and local smoothing effects to the Euclidean case under cylindrical symmetry (Theorems \ref{theorem_resolvent_radial} and \ref{theorem_smoothing_radial}). 

\subsection{Organization of the paper}
The rest of the paper is devoted to the proof of Theorems \ref{theorem_resolvent_general} and \ref{theorem_resolvent_radial}. The main tool, and key ingredient in this paper, is an abstract uniform resolvent estimate, Theorem \ref{theorem_abstract}, which is proved in Section~\ref{section_abstract} via the method of weakly conjugate operators due to Boutet de Monvel--M\u{a}ntoiu \cite{BM} and Hoshiro \cite{Hoshiro}, see also \cite{Richard}. The latter can be regarded as a variation of Mourre's conjugate operators method. From these abstract estimates, we will deduce, in Section~\ref{sec:proof_thm}, our main results after proving the uniform boundedness of some weighted operators; Hardy inequalities on the Heisenberg group will be also used. In Appendix A, we briefly recall the definition and basic notions of the fractional sublaplacians $\mathcal L^s$ and $\mathcal L_s$.  In Appendix B, we record basic notions of the joint spectral theory for two commuting self-adjoint operators, used in the proof.

\subsection{Notation} 
\label{sub:notation}
We list here some notations we will use throughout the paper:
\begin{itemize}
\item $\S(\mathbb H^d)=\S(\R^{2d+1})$ denotes the standard Schwartz class on $\R^{2d+1}$.

%\item $L^2=L^2(\mathbb H^d)$ is the Hilbert space with inner product and norm$$\<f,g\>=\int_{\mathbb H^d} f(z,t)\overline{g(z,t)} \ dzdt,\quad \|f\|:=\|f\|_{L^2(\mathbb{H}^d)}=\sqrt{\<f,f\>}.$$

\item $H^s=H^{s}(\mathbb H^d)=\<\L\>^{-s/2}L^2(\mathbb H^d)$ is the Sobolev space equipped with norm $$\|f\|_{H^s}=\|\<\L\>^{s/2}f\|_{L^2(\mathbb H^d)},$$ where $\<\L\>^{s/2}$ is defined via the spectral decomposition theorem (see \eqref{spectral_decomposition_theorem_1}). 

\item $\mathbb B(U,V)$ is the space of bounded operators from $U$ to $Y$, $\norm{\ \cdot\ }_{\mathbb{B}(U,V)}$ denotes its norm, and $\mathbb B(U):=\mathbb B(U,U)$. In particular, we will denote
\begin{equation*}
	\norm{A} := \norm{A}_{\mathbb{B}(L^2(\mathbb{H}^d))}
\end{equation*} 
for any $A \in \mathbb{B}(L^2(\mathbb{H}^d))$.

\item Given a densely defined operator $A$, its closure will be denoted by $\overline{A}$.

\item For two closed densely defined operators $A$ and $B$, $[A,B]$ denotes the commutator of $A$ and $B$: for $f\in D(A)\cap D(B)$ and $g\in D(A^*)\cap D(B^*)$,
$$
\<[A,B]f,g\>:=\<Bf,A^*g\>-\<Af,B^*g\>. 
$$

\item Given two operators $A$ and $B$, $A\subset B$ means $B$ is an extension of $A$, namely $D(A)\subset D(B)$ and $B=A$ on $D(A)$. 
\end{itemize}

%section
\section{An abstract uniform resolvent estimate}
\label{section_abstract}

In this section, we study a uniform resolvent estimate in an abstract setting, which will play an essential role in the proof of the main theorems. 

We first recall basic facts on the functional calculus for self-adjoint operators used in the paper (see \cite{ReSi} on the spectral theory for self-adjoint operators). For any self-adjoint operator $\mathcal{H}$ on $L^2(\mathbb H^d)$ and any Borel measurable function $\varphi(\lambda)$ on $\R$ which is finite almost everywhere, a densely defined closed operator $\varphi(\mathcal{H})$ is defined via the spectral decomposition theorem as 
	\begin{equation}
		\label{spectral_decomposition_theorem_1}
		\begin{split}
			\<\varphi(\mathcal{H})f,g\> &:=\int_\R \varphi(\lambda) d\<E_\mathcal{H}(\lambda)f,g\>
			\\
			D(\varphi(\mathcal{H})) &:=\big\{f\in L^2(\mathbb H^d)\ |\ \int_\R |\varphi(\lambda)|^2 d\<E_\mathcal{H}(\lambda)f,f\><\infty\big\},
		\end{split}
	\end{equation}
	where $E_\mathcal{H}$ denotes the spectral measure associated with $\mathcal{H}$. Here we list a few basic properties: 
\begin{itemize}
\item $\supp E_\mathcal{H}=\sigma(\mathcal{H})$;
\item $D(\varphi_1(\H)+\varphi_2(\H))=D(\varphi_1(\H))\cap D(\varphi_2(\H))$ and $\overline{\varphi_1(\H)+\varphi_2(\H)}=(\varphi_1+\varphi_2)(\H)$. In particular, they coincide with each other if one between $\varphi_1$ and $\varphi_2$ is bounded;
\item $D(\varphi_1(\H)\varphi_2(\H))=D((\varphi_1\varphi_2)(\H))\cap D(\varphi_2(\H))$ and $\overline{\varphi_1(\H)\varphi_2(\H)}=(\varphi_1\varphi_2)(\H)$. In particular, if $\varphi_2$ is bounded then $\varphi_1(\H)\varphi_2(\H)=(\varphi_1\varphi_2)(\H)$; 
\item $\varphi(\mathcal{H})^*=\overline{\varphi}(\mathcal{H})$. In particular, $\varphi(\mathcal{H})$ is self-adjoint if $\varphi$ is real-valued;
\item If $\varphi\in L^\infty(\R)$, then $\varphi(\mathcal{H})\in \mathbb B(L^2)$ and
\begin{align}
\label{spectral_decomposition_theorem_2}
\|\varphi(\mathcal{H})\|\le \|\varphi\|_{L^\infty(\R)}. 
\end{align}
\end{itemize}

Let $Q=2d+2$ be the homogeneous dimension of $\mathbb H^d$ and 
\begin{equation}
\label{eq:A}
A:=-i \left( z\cdot \partial_z + 2t\partial_t + \frac{Q}{2} \right)=-i z\cdot \nabla_{\mathbb H}-i2t\partial_t-i(d+1)
\end{equation}
be the self-adjoint generator of the dilation group
$$
e^{i\tau A}f(z,t)=e^{\frac Q2 \tau}f(e^\tau z,e^{2\tau}t),\quad \tau\in \R. 
$$
The operator $e^{i\tau A}$ is unitary on $L^2(\mathbb H^d)$. Moreover, it is easy to see that $e^{-i\tau A} \L e^{i\tau A}=e^{2\tau} \L$ (see Lemma\til\ref{lemma_abstract_commutator_1} below). In particular, $e^{i\tau A}$ leaves $H^m(\mathbb H^d)$ invariant for any $m\in \R$. 

From this point on, we will consider $\H$ to be a non-negative self-adjoint operator on $L^2(\mathbb H^d)$, with domain $H^{2s}(\mathbb H^d)$, satisfying the conditions \ref{H1}, \ref{H2} and \ref{H3} in Section \ref{main_result}.

\begin{remark}
	\label{rem:H3}
	Hypothesis \ref{H3} is necessary in order to ensure a joint spectral measure associated to $\L$ and $\H$, and to define the operator $\phi(\L,\H)$ for any measurable, almost everywhere finite function $\phi\colon\R^2 \to \C\cup\{\infty\}$. We refer to Appendix\til\ref{appendix_joint_spectral_theory} for more details. It follows that if $\phi_1,\phi_2:\R\to \C$ are of polynomial growth, then $[\phi_1(\L),\phi_2(\H)]=0$ on $\mathcal S(\mathbb H^d)$. In particular, \ref{H3} implies $\mathcal H$ and $\L$ commute.
\end{remark}

Let us define the operators
\begin{equation}\label{def:SA}
	S_\alpha:=\L^{s/2}\<\L\>^{-\alpha/2},
	\qquad
	A_\alpha:=\<\L\>^{-\alpha/2}A\<\L\>^{-\alpha/2},
\end{equation}
for any $\alpha\ge0$.

%theorem
\begin{theorem}	
\label{theorem_abstract}
Let $s>0$, $1/2<\mu\le1$ and $\alpha\ge0$ and assume \ref{H1},\til\ref{H2} and\til\ref{H3}. Then
\begin{align}
\label{theorem_abstract_1}
\sup_{\sigma\in \C\setminus\R}\|\<A_\alpha\>^{-\mu}S_\alpha(\H-\sigma)^{-1}S_\alpha\<A_\alpha\>^{-\mu}\|\le \kappa(\H,s,\mu),
\end{align}
where $\kappa(\H,s,\mu)$ is given by
	\begin{equation}\label{def:kappa}
	\kappa(\H,s,\mu)=
	\inf_{b>0}
	e^{2b}
	\left\{
	K_1 b^{-1/2}
	+
	K_2 b^{1/2} D\left(2,b^{1/2}\right)
	+
	K_3
	b^{\mu-1/2} D\left(\frac{1}{\mu-1/2} , b^{\mu-1/2} \right)	
	\right\}^2
\end{equation}
where $K_1:=(2C_1^{2}C_2^{-1})^{-1/2}$, $K_2:= (2C_1^{-1}s)^{1/2} \max\{s,2e^{-(1+s/2)}\}$, $K_3:=\frac{2-\mu}{\mu-1/2} (2C_1s)^{-1/2}$ and 
\begin{equation}
	\label{def:dawson}
	D(p,x) :=e^{-x^p} \int_{0}^{x} e^{\tau^p} d\tau
\end{equation}
is the generalized Dawson's integral.

%
\begin{comment}
	$$
	\kappa(\H,s,\mu)=
	\inf_{b>0}
	\frac{e^{C_2 b}}{b}
	\left\{
	C_1^{-1}
	+
	\max\{s,2e^{-(1+s/2)}\}
	b
	+
	\frac{2-\mu}{2\mu-1} \left(\frac{2}{C_1}\right)^{1-\mu} s^{-\mu}
	b^{\mu}
	\right\}^2
	$$
\end{comment}
\end{theorem}

\begin{remark}
\label{remark_kappa}
%	Via Kummer's transformation, the Dawson's integral can be written in terms of the Kummer's confluent hypergeometric function	$$	D(p,x) = x e^{-x^p} M\left( \frac1p , \frac1p+1 , x^p \right).	$$
Expanding the Maclaurin series in \eqref{def:dawson}, one obtains the bound 
	$
	D(p,x) < x^{1-p}(1-e^{-x^p})
	$ for $p\ge1$ and $x>0$, with which we can give a more explicit estimate for $\kappa(\H,s,\mu)$:
	\begin{equation*}
		\kappa(\H,s,\mu) < \inf_{b>0} \left\{ K_1 b^{-1/2} e^b +(K_2 + K_3 b^{2\mu-2}) (e^b-1) \right\}^2.
	\end{equation*}
Moreover, in case of $\H=\mathcal L$, $C_1=C_2=1$ and one can thus calculate an explicit upper bound of $\kappa(\H,s,\mu)$. For instance, if $s=\mu=1$	then
$$
\kappa(\L,1,1)
=
\inf_{b>0}
\frac{e^{2b}}{2b} \left\{ 1 + 4b \ D(2,b^{1/2}) \right\}^2
\approx
6.42686
< 2(3e^{1/4}-2)^2.
$$
Note that, although the constant $\kappa(\H,s,\mu)$ seems to be far from optimal, obtaining an optimal, or even explicit, constant in uniform resolvent estimates is typically challenging, except for some special cases such as the Laplacian $\Delta$ on $\mathbb{R}$ and $\mathbb{R}^3$. It is also worth remarking that there is extensive literature on the best constant for smoothing effects in the case of $e^{it\Delta_{\mathbb{R}^d}}u_0$ (see for instance \cite{Simon}), but there appear to be no results on the best constant (or even an explicit expression of the constant) for smoothing effects in the case of the inhomogeneous propagator $\int_0^t e^{i(t-s)\Delta_{\mathbb{R}^d}}F(s)ds$.
\end{remark}

\begin{remark}
Theorem \ref{theorem_abstract} is useful for studying, in a unified way, resolvent estimates of the form $$\sup_{\sigma\in \C\setminus\R}\|w a(\L)(\H-\sigma)^{-1}a(\L)w\|<\infty$$ with some weight function $w=w(z,t)$ and Sobolev weight $a(\L)$. Namely, this theorem allows us to reduce the proof of \eqref{theorem_resolvent_general_1} into showing that $\<A_\alpha\>^{\mu} S_\alpha^{-1} a(\L) w\in \mathbb B(L^2)$ for suitable choices of $a$ and $w$, and this operator is independent of $\H$.
\end{remark}

For the proof of Theorem \ref{theorem_abstract}, we prepare several lemmas. The following is useful to calculate various commutators against $iA$ and will be frequently used throughout the paper. 

%{lemma}
\begin{lemma}
\label{lemma_abstract_commutator_1}
Let $\varphi \in C([0,\infty))$ be such that $|\varphi(\lambda)|\lesssim\<\lambda\>^r$ with some $r\in \R$. Then
\begin{align}
\label{lemma_abstract_commutator_1_1}
e^{-i\tau A}\varphi(\H) e^{i\tau A}=\varphi(e^{2s\tau}\H)\quad\text{on}\quad H^{2sr}(\mathbb H^d).
\end{align}
If in addition $\varphi \in C([0,\infty)) \cap C^1((0,\infty))$ and $|\lambda \varphi'(\lambda)|\lesssim\<\lambda\>^r$, then
\begin{align}
\label{lemma_abstract_commutator_1_2}
[\varphi(\H),iA]=2s\H\varphi'(\H)\quad\text{on}\quad H^{2sr}(\mathbb H^d).
\end{align}
In particular, 
\begin{align*}
%\label{lemma_abstract_commutator_1_3}
[\H,iA]=2s\H,\quad [[\H,iA],iA]=4s^2\H\quad\text{on}\quad H^{2s}(\mathbb H^d).
\end{align*}
\end{lemma}

%remark
\begin{remark}
Precisely speaking, \eqref{lemma_abstract_commutator_1_2} means that $[\varphi(\H),iA]$ defined originally as a quadratic form on $D(A)\cap H^{2sr}(\mathbb H^d)$ extends to a self-adjoint operator, which coincides with $2s\H\varphi'(\H)$. 
{\color{purple}

}
\end{remark}

%proof
\begin{proof}[Proof of Lemma \ref{lemma_abstract_commutator_1}]
The formula \eqref{homogeneous} implies $$e^{-i\tau A}(\H-\sigma)^{-1} e^{i\tau A}=(e^{2s\tau}\H-\sigma)^{-1}$$ for all $\sigma\in \C\setminus[0,\infty)$. Let $E_\H$ be the spectral measure of $\H$. By using Stone's formula: 
\begin{align*}
%\label{Stone}
E_\H((a,b))=\frac{1}{2\pi i}\lim_{\ep\searrow0}\int_a^b \{(\H-\lambda-i\ep)^{-1}-(\H-\lambda+i\ep)^{-1}\}d\lambda
\end{align*}
and a standard limiting argument, we thus obtain for any Borel measurable set $\Omega\subset\R$, 
$$
e^{-i\tau A}E_\H(\Omega) e^{i\tau A}=E_{e^{2s\tau}\H}(\Omega)
$$ which, combined with the formula \eqref{spectral_decomposition_theorem_1}, implies \eqref{lemma_abstract_commutator_1_1}. 
We also obtain 
\begin{align*}
\<[\varphi(\H),iA]f,g\>
&=\frac{d}{d\tau}\<e^{-i\tau A}\varphi(\H)e^{i\tau A}f,g\>|_{\tau=0}=\frac{d}{d\tau}\<\varphi(e^{2s\tau}\H)f,g\>|_{\tau=0}\\
&=\frac{d}{d\tau}\Big |_{\tau=0}\int_0^\infty \varphi(e^{2s\tau}\lambda) \, d\<E_{\H}(\lambda)f,g\>
\\
&=\lim_{\tau\searrow0} \int_0^\infty 2s e^{2s\tau_0}\lambda \varphi'(e^{2s\tau_0}\lambda) \, d\<E_{\H}(\lambda)f,g\>
\end{align*}
for $f\in H^{2sr}(\mathbb H^d)$, $g\in L^2(\mathbb H^d)$ and some $\tau_0\in(0,\tau)$, where the last step is justified by the mean value theorem. Since $\tau$ is small and using \eqref{spectral_decomposition_theorem_1}, we have
$$
\int_\R 2s e^{2s\tau_0}\lambda \varphi'(e^{2s\tau_0}\lambda) \, d\<E_{\H}(\lambda)f,g\>
\lesssim
\int_\R \<\lambda\>^r \, d\<E_{\H}(\lambda)f,g\>
\le
\norm{\<\H\>^r f} \norm{g}
< \infty.
$$ 
Therefore by dominated convergence theorem we obtain
$$
\<[\varphi(\H),iA]f,g\>=\int_\R 2s\lambda \varphi'(\lambda)d\<E_{\H}(\lambda)f,g\>
=\<2s\H\varphi'(\H)f,g\>
$$
and \eqref{lemma_abstract_commutator_1_2} follows. 
\end{proof}

	\begin{remark}
		In the proof of Lemma\til\ref{lemma_abstract_commutator_1} we actually explicitly use only the polynomial bound on $\lambda\varphi'(\lambda)$. However, the condition $|\varphi(\lambda)|\le\jap{\lambda}^r$ is stated to ensure that the domain of $\varphi(\H)$ contains $\mathcal S(\mathbb H^d)$. 
	\end{remark}

%{lemma}\begin{lemma}\label{lemma_abstract_domain_1}Let $Q\in \mathbb B(L^2)$ be such that $[A,Q]$ extends to a bounded operator $[A,Q]_0$ on $L^2$. Then $Q$ leaves $D(A)$ invariant so that$AQ=QA+[A,Q]_0$ on $D(A)$. \end{lemma}\begin{proof}Let $f,g\in D(A)$. Then $h:=([A,Q]_0+QA)f\in L^2$ by the assumption and\begin{align*}\<Qf,Ag\>=\<Qf,Ag\>-\<Af,Q^*g\>+\<Af,Q^*g\>=\<h,g\>.\end{align*}Since $A$ is self-adjoint, this means $Qf\in D(A^*)=D(A)$ and $AQf=h$. \end{proof}

%{lemma}
\begin{lemma}
\label{lemma_abstract_domain}
Let $\varphi\in C([0,\infty))\cap C^1((0,\infty))$ be such that $\varphi(\lambda)$ and $\lambda \varphi'(\lambda)$ are bounded on $[0,\infty)$. Then $\varphi(\H)$ leaves $D(A)$ invariant and $$A\varphi(\H)=\varphi(\H)A+2is\H\varphi'(\H)$$ on $D(A)$. 
\end{lemma}

\begin{proof}
By Lemma \ref{lemma_abstract_commutator_1} with $r=0$, $[A,\varphi(\H)]=2is\H\varphi'(\H)\in \mathbb B(L^2)$. Let $f,g\in D(A)$. Then
\begin{align*}
\<\varphi(\H)f,Ag\>
&=\<\varphi(\H)f,Ag\>-\<Af,\varphi(\H)^*g\>+\<Af,\varphi(\H)^*g\>\\
&=\<[A,\varphi(\H)]f,g\>+\<Af,\varphi(\H)^*g\>\\
&=\<2is\H\varphi'(\H)f,g\>+\<Af,\varphi(\H)^*g\>.
\end{align*}
Since $2is\H\varphi'(\H)f+\varphi(\H)Af\in L^2$, this equality means $\varphi(\H)f\in D(A^*)=D(A)$ and $A\varphi(\H)f=2is\H\varphi'(\H)f+\varphi(\H)Af$ by the definition of $A^*$ and the self-adjointness of $A$. 
\end{proof}

%remark
%\begin{remark}
%Lemmas \ref{lemma_abstract_commutator_1} and \ref{lemma_abstract_domain} also hold with $\H$ replaced by any homogeneous self-adjoint operator of degree $2s$ on $L^2(\mathbb H^d)$ with domain $H^{2s}(\mathbb H^d)$. 
%\end{remark}

Recall the definition of $A$ in \eqref{eq:A} and $A_{\alpha}$ in \eqref{def:SA}.
%lemma
\begin{lemma}
\label{lemma_B}
$A$ is essentially self-adjoint on $\mathcal S(\mathbb H^d)$.
Moreover, for any $\alpha>0$, $A_\alpha$ is essentially self-adjoint on $D(A)$.
\end{lemma}

%proof
\begin{proof}

The assertion for $A$ follows from the Nelson's Criterion (see \cite[Proposition 5.3]{Amrein}).
By Lemma\til\ref{lemma_abstract_domain}, $\<\L\>^{-\alpha/2}D(A)\subset D(A)$ and $A_\alpha$ thus makes sense on $D(A)$. To prove the essential self-adjointness of $A_\alpha$, it is enough to check $\overline{\Ran(A_\alpha\pm i\sigma)}=L^2(\mathbb H^d)$ for some $\sigma>0$. By the essential self-adjointness of $A$, $\Ran(A\pm i\sigma)^\perp=\{0\}$ for any $\sigma>0$. Then
\begin{align*}
0&=\<f,(A_\alpha\pm i\sigma )g\>=\<f,(\<\L\>^{-\alpha/2}A\<\L\>^{-\alpha/2}\pm i\sigma)g\>\\
&=\<\<\L\>^{-\alpha/2}f,(A\pm i\sigma)\<\L\>^{-\alpha/2}g\>\mp i\sigma \<f,(1-\<\L\>^{-\alpha/2})g\>
%&=\<\<\H\>^{-\alpha/2}f,(A\pm i\sigma)\<\H\>^{-\alpha/2}g\>-i\sigma \<\<\H\>^{-\alpha/2}f,(1-\<\H\>^{-\alpha/2})g\>
\end{align*}
for $f\in \Ran(A_\alpha\pm i\sigma)^\perp\cap D(A)$ and $g\in D(A)$. Letting $g=f$ and taking the real parts imply $$\Re \<\<\L\>^{-\alpha/2}f,(A\pm i\sigma)\<\L\>^{-\alpha/2}f\>=0.$$ Similarly, letting $g=if$ and taking the imaginary part imply $$\Im\<\<\L\>^{-\alpha/2}f,(A\pm i\sigma)\<\L\>^{-\alpha/2}f\>=0.$$ Thus $\<\L\>^{-\alpha/2}f$ belongs to $\Ran(A\pm i\sigma)^\perp$ and vanishes identically. Since $\<\L\>^{-\alpha/2}$ is positive, $f\equiv0$. This shows $\Ran(A_\alpha\pm i\sigma)^\perp=\{0\}$ since $D(A)$ is dense in $L^2(\mathbb{H}^d)$ and $\Ran(A_\alpha\pm i\sigma)^\perp$ is closed. Hence $\overline{\Ran(A_\alpha\pm i\sigma)}=L^2(\mathbb H^d)$. 
\end{proof}

%\begin{color}{teal}[[\marginnote{Complete proof, can be omitted}
%H: The claim for $A$ in Lemma \ref{lemma_B} has been modified. I found a simpler criterion (also called Nelson's Criterion) in the book [Proposition 5.3, Amrein (Hilbert Space Methods in Quantum Mechanics)].
%\textit{Proof:}
%Let $f\in \Ker (A^*-i)$ and $g\in \mathcal S(\mathbb H^d)$. Then, since $ e^{-i\tau A}\mathcal S(\mathbb H^d)\subset \mathcal S(\mathbb H^d)$, $v(\tau):=\<f,e^{-i\tau A}g\>$ solves
%$$
%\frac{d}{d\tau}v(\tau)=i\<A^*f,e^{-i\tau A}g\>=-v(\tau),
%$$
%so $v(t)=v(0)e^{-\tau}$. If $v(0)\neq0$, $|v(\tau)|\to \infty$ and $\tau\to -\infty$, contradicting with the uniform bound $|v(\tau)|\le \|f\|\|g\|$. Thus, $v(0)=\<f,g\>=0$ and $f=0$ since $\mathcal S(\mathbb H^d)$ is dense in $L^2(\mathbb H^d)$. Similarly $\Ker (A^*+i)=\{0\}$ and thus $A$ is essentially self-adjoint on $\mathcal S(\mathbb H^d)$. ]]
%\end{color}

From now on, we use the same symbol $A_\alpha$ for its self-adjoint extension. Applying Lemma \ref{lemma_abstract_domain} repeatedly, we have $\varphi(\H)D(A_\alpha)\subset D(A_\alpha)$ and $\varphi(\L)D(A_\alpha)\subset D(A_\alpha)$ for $\varphi$ satisfying the same conditions as in Lemma \ref{lemma_abstract_domain}. 

%{lemma}
\begin{lemma}
\label{lemma_W}
Let $0<\mu\le1$ and
$$
%	\label{def:We}
	W_\ep=\<A_\alpha\>^{-\mu}\<\ep A_\alpha\>^{\mu-1},
$$
where $A_\alpha$ is defined in \eqref{def:SA}.
Then the map $\ep\mapsto W_\ep \in \mathbb B(L^2(\mathbb H^d))$ is continuous on $[0,\infty)$ and $C^1$ on $(0,\infty)$. Moreover, 
$$
\norm{W_\ep}\le 1
$$
for all $\ep\ge0$ and 
$$
\norm{W_\ep'}\le (1-\mu) \ep^{\mu-1},\quad \norm{A_\alpha W_\ep }\le \ep^{\mu-1}
$$
for all $\ep>0$, where 
$W_\ep':=\frac{d}{d\ep}W_\ep=(\mu-1)\ep\<A_\alpha\>^{-\mu}|A_\alpha|^2\<\ep A_\alpha\>^{\mu-3}
$. 
\end{lemma}

%proof
\begin{proof}
It is easy to see that the symbols of $W_\ep,W_\ep'$ and $A_\alpha W_\ep$ satisfy 
\begin{align*}
\<\lambda\>^{-\mu}\<\ep \lambda\>^{\mu-1}&\le1,\\
\ep\<\lambda\>^{-\mu}|\lambda|^2\<\ep \lambda\>^{\mu-3}&=\ep^{\mu-1}\cdot |\lambda|^\mu\<\lambda\>^{-\mu}\cdot (\ep|\lambda|)^{-\mu+2}\<\ep \lambda\>^{\mu-2}\cdot \<\ep \lambda\>^{-1}\le \ep^{\mu-1},\\
|\lambda|\<\lambda\>^{-\mu}\<\ep \lambda\>^{\mu-1}&=\ep^{\mu-1}\cdot |\lambda|^\mu\<\lambda\>^{-\mu}\cdot (\ep|\lambda|)^{-\mu+1}\<\ep \lambda\>^{\mu-1}\le \ep^{\mu-1}.
\end{align*}
The lemma follows from these estimates and the spectral decomposition theorem. 
\end{proof}

%lemma
\begin{lemma}[Gronwall's inequality]
\label{lemma_Gronwall}
Let $f\in C(a,b)$ and $u_1,u_2 \in L^1(a,b)$ be non-negative functions. Suppose that for some constant $A\ge0$ and for all $a<\ep<b$, 
$$
f(\ep)\le A+\int_\ep^b \Big(u_1(r)f(r)+u_2(r)f(r)^{1/2}\Big) dr. 
$$
Then, for all $a<\ep<b$, 
$$
f(\ep)\le \left\{A^{1/2}+\frac12\int_\ep^bu_2(r) \exp\left(-\frac12\int_r^b u_1(\tau)d\tau\right)dr\right\}^2\exp\left(\int_\ep^bu_1(\tau)d\tau\right).
$$
\end{lemma}

\begin{proof}
Let $U(\ep):=u_2(\ep)\exp\left(-\frac12\int_\ep^bu_1(\tau)d\tau\right)$, $g(\ep):=A+\int_\ep^b \Big(u_1(r)f(r)+u_2(r)f(r)^{1/2}\Big) dr$ and $h(\ep):=g(\ep)\exp\left(-\int_\ep^bu_1(\tau)d\tau\right)$. It is easy to prove the inequality $\partial_\ep\sqrt{h(\ep)}\ge -U(\ep)/2$. The lemma follows by integrating this inequality over the interval $(\ep,b)$ and noticing that $f	\le g$. 
\end{proof}

\begin{proof}[Proof of Theorem \ref{theorem_abstract}]
The proof is decomposed into several steps. 

{\it Step 1}. We first compute the commutators $[\H,iA_\alpha]$ and $[[\H,iA_\alpha],iA_\alpha]$. %Since $\H$ commutes with $\L$, we can check by a similar argument as that in the proof of Lemma \ref{lemma_abstract_commutator_1} that $\H$ also commutes with $\<\L\>^{-\alpha}$. Precisely,  it follows from a direct calculation that, for any $\sigma \in \C\setminus\R$, $\H(\L-\sigma)^{-1}=(\L-\sigma)^{-1}\H$. By Stone's formula \eqref{Stone} and a limiting argument, we also have $\H E_\L(\Omega)=E_\L(\Omega)\H$ for any Borel set $\Omega\subset \R$. Therefore, \eqref{spectral_decomposition_theorem_1} implies that $\H$ commutes with $\varphi(\L)$ for any $\varphi\in C(\R)$ of polynomial growth. 
By \ref{H3} and Remark\til\ref{rem:H3}, $\H$ commutes with $\varphi(\L)$ for any $\varphi\in C(\R)$ of polynomial growth. Using Lemma\til\ref{lemma_abstract_commutator_1}, we then calculate
\begin{align}
\label{theorem_abstract_proof_-2}
[\H,iA_\alpha]&=\<\L\>^{-\alpha/2}[\H,iA]\<\L\>^{-\alpha/2}=2s\H\<\L\>^{-\alpha},\\
\nonumber
[[\H,iA_\alpha],iA_\alpha]&=2s\<\L\>^{-\alpha}\H\<\L\>^{-\alpha/2}iA \<\L\>^{-\alpha/2}-2s\<\L\>^{-\alpha/2}iA\<\L\>^{-\alpha/2}\H\<\L\>^{-\alpha}\\
%&=\<\L\>^{-\alpha}[\H,iA]\<\L\>^{-\alpha/2}iA \<\L\>^{-\alpha/2}-\<\L\>^{-\alpha/2}iA\<\L\>^{-\alpha/2}[\H,iA]\<\L\>^{-\alpha}\\
\nonumber
&=2s\<\L\>^{-\alpha}[\H,iA] \<\L\>^{-\alpha}+2s\<\L\>^{-\alpha}\H[\<\L\>^{-\alpha/2},iA]\<\L\>^{-\alpha/2}\\&\quad \nonumber
-2s\<\L\>^{-\alpha/2}[iA,\<\L\>^{-\alpha/2}]\H\<\L\>^{-\alpha}\\
\nonumber
&=4s^2\<\L\>^{-\alpha}\H\<\L\>^{-\alpha}-4s^2\alpha \<\L\>^{-\alpha}\H\L^2\<\L\>^{-2}\<\L\>^{-\alpha}\\
\label{theorem_abstract_proof_-1}
&=4s^2\H\<\L\>^{-2\alpha}(1 -\alpha \L^2\<\L\>^{-2})
\end{align}
as quadratic forms on $H^{2s}(\mathbb H^d)\cap D(A_\alpha)$. In particular, $[\H,iA_\alpha]$ and $[[\H,iA_\alpha],iA_\alpha]$ extend to self-adjoint operators with domains which include $H^{2s-\alpha}(\mathbb H^d)$. %For any $\alpha>0$, $[\H,iA_\alpha]$ is infinitesimally relative $\H$-bounded, since $\<\L\>^{-s} \H=\H\<\L\>^{-s}\in \mathbb B(L^2(\mathbb H^d))$ by Remark\til\ref{rem:H3} and, for any $\gamma>0$, \begin{align*}\|\<\L\>^{-\alpha}\H f\|&\le \|E_{\L}([\gamma,\infty))\<\L\>^{-\alpha}\|\|\H f\|+\|E_{\L}([0,\gamma))\<\L\>^{s-\alpha}\|\|\<\L\>^{-s}\H f\|\\&\le \<\gamma\>^{-\alpha}\|\H f\|+C_\gamma \|f\|. \end{align*}
Moreover, since $\jap{\lambda}^{-\alpha}|1 -\alpha\lambda^2\<\lambda\>^{-2}|\le 1$ for any $\alpha\ge0$, \ref{H2} and \eqref{spectral_decomposition_theorem_2} imply
\begin{align}
\label{theorem_abstract_proof_0}
|\<[[\H,iA_\alpha],iA_\alpha]f,f\>|\le 4C_2s^2 \|S_\alpha f\|^2. 
\end{align}

{\it Step 2}. Let $\sigma_1\in \R$, $\sigma_2,\ep>0$ and consider the operator $\H^\pm_\ep \colon H^{2s}(\mathbb H^d)\to L^2(\mathbb H^d)$ defined as
\begin{equation}
	\label{def:Hep}
	\H^\pm_\ep := \H-(\sigma_1 \pm i\sigma_2) \mp i \ep [\H,iA_\alpha]
	= 
	\H-\sigma_1\mp i\sigma_2\mp i 2s\ep\H\<\L\>^{-\alpha}.
\end{equation}
We shall show that $\H^\pm_\ep$ is bijective and $(\H^\pm_\ep)^*=\H^\mp_\ep$. It is easy to see that $\Ker \H^\pm_\ep=\{0\}$. Indeed, since $\H$ and $\<\L\>^{-\alpha}$ are non-negative, self-adjoint and commute with each other, we have $\Im\<\H f,f\>=0$, 
\begin{align*}
&\Im\<f,\H\<\L\>^{-\alpha}f\>=\Im \<\H\<\L\>^{-\alpha/2}f,\<\L\>^{-\alpha/2}f\>=0,\\
&\Im \<\H f,\H\<\L\>^{-\alpha}f\>=\Im \<\H\<\L\>^{-\alpha/2}f,\H\<\L\>^{-\alpha/2}f\>=0,
\end{align*}
and $\ep s \sigma_2 \<\H\<\L\>^{-\alpha}f,f\> \ge0$
for $f\in H^{2s}(\mathbb H^d)$. Hence, 
\begin{align}
	\label{eq:bound_below_Hpm}
\nonumber
\|\H^\pm_\ep f\|^2%\ge \| \sigma_2f + 2s\ep \H\<\L\>^{-\alpha} f\|^2
&=\|(\H-\sigma_1)f\|^2\mp 2\Im\<(\H-\sigma_1)f,(\sigma_2+2s\ep\H\<\L\>^{-\alpha})f\>+\|(\sigma_2+ 2s\ep\H\<\L\>^{-\alpha})f\|^2\\
&\ge \max\{\sigma_2^2\|f\|^2,4s^2\ep^2\|\H\<\L\>^{-\alpha}f\|^2\},
\end{align}
which particularly implies $\Ker \H^\pm_\ep=\{0\}$. 
To show $\Ran \H^\pm_\ep=L^2(\mathbb H^d)$ and $(\H^\pm_\ep)^*=\H^\mp_\ep$, we use some tools from the joint spectral theory for strongly commuting self-adjoint operators (see Appendix \ref{appendix_joint_spectral_theory} below). $\H^\pm_\ep$ is written as $\H^\pm_\ep=\varphi_1(\H)\pm i\varphi_{2,\alpha}(\H,\L)$ with two real-valued functions $\varphi_1(\lambda_1)=\lambda_1-\sigma_1$ and $\varphi_{2,\alpha}(\lambda_1,\lambda_2)=-\sigma_2-2\ep s\lambda_1\<\lambda_2\>^{-\alpha}$. Since $|(\varphi_1\pm i\varphi_{2,\alpha})(\lambda_1,\lambda_2)|^2=|\varphi_1(\lambda_1)|^2+|\varphi_{2,\alpha}(\lambda_1,\lambda_2)|^2$, we have
$$
\|(\varphi_1\pm i\varphi_{2,\alpha})(\H,\L)f\|^2=\|\varphi_1(\H)f\|^2+\|\varphi_{2,\alpha}(\H,\L)f\|^2=\|\H^\pm_\ep f\|^2.
$$
Since $(\varphi_1\pm i\varphi_{2,\alpha})(\H,\L)$ is the closure of $\varphi_1(\H)\pm i\varphi_{2,\alpha}(\H,\L)$ by Proposition \ref{proposition_joint_1}, this equality shows $D(\H^\pm_\ep)=D((\varphi_1\pm i\varphi_{2,\alpha})(\H,\L))$ and thus $\H^\pm_\ep=(\varphi_1\pm i\varphi_{2,\alpha})(\H,\L)$. In particular, $\H^\pm_\ep$ is closed and 
$$
(\H^\pm_\ep)^*=[(\varphi_1\pm i\varphi_{2,\alpha})(\H,\L)]^*=\overline{(\varphi_1\pm i\varphi_{2,\alpha})}(\H,\L)=(\varphi_1\mp i\varphi_{2,\alpha})(\H,\L)=\H^\mp_\ep
$$
Thus, $\Ran\H^\pm_\ep=(\Ker \H^\mp_\ep)^\perp=\{0\}$ and $\H^\pm_\ep$ is bijective. 

%For $\alpha>0$, since $\ep[\H,iA_\alpha]$ is infinitesimally relative $\H$-bounded, one can follow the same argument as in the proof of the Kato--Rellich's theorem (see \cite[Theorem~X.12]{ReSi2}) to find $\Ran \H^\pm_\ep=L^2(\mathbb H^d)$. It remains to show $(\H^\pm_\ep)^*=\H^\mp_\ep$. It is easy to see that $\H^\mp_\ep \subset(\H^\pm_\ep)^*$. Moreover, for $u\in D((\H^\pm_\ep)^*)$, there exists $v\in D(\H^\mp_\ep)$ such that $(\H^\pm_\ep)^*u=\H^\mp_\ep v=(\H^\pm_\ep)^*v$. Hence, for any $w\in L^2(\mathbb{H}^d)$, there exists $\tilde w\in D(\H^\pm_\ep)$ with $w=\H^\pm_\ep\tilde w$ such that$$\<u-v,w\>=\<u-v,\H^\pm_\ep\tilde w\>=\<(\H^\mp_\ep)^*(u-v),\tilde w\>=0,$$which implies $u=v\in D(\H^\mp_\ep)$ and thus $(\H^\pm_\ep)^*=\H^\mp_\ep$. Hence $\H^\pm_\ep$ are also bijective for $\alpha>0$. 

Let $G^\pm_\ep=(\H^\pm_\ep)^{-1} \colon L^2(\mathbb H^d)\to H^{2s}(\mathbb H^d)$ be the inverse of $\H^\pm_\ep$ such that $(G^\pm_\ep)^*=G^\mp_\ep$ and 
\begin{align}
\label{eq:G_ep}
\|G^\pm_\ep\|\le \sigma_2^{-1}
\end{align}
by \eqref{eq:bound_below_Hpm}. By the duality and interpolation, we also have $G^\pm_\ep\in %\mathbb B(L^2(\mathbb H^d),H^{2s}(\mathbb H^d))\cap \mathbb B(H^{-2s}(\mathbb H^d),L^2(\mathbb H^d))\cap 
\mathbb B(H^{-s}(\mathbb H^d),H^{s}(\mathbb H^d))$. 
%In particular, since $D(S_\alpha)\subset H^{s}(\mathbb H^d)$ and $D(\<\mathcal H\>^{1/2})=H^{s}(\mathbb H^d)$ by \ref{H2}, $S_\alpha G^\pm_\ep S_\alpha$ and $\<\H\>^{1/2}G^\pm_\ep\<\H\>^{1/2}$  are bounded on $L^2(\mathbb H^d)$, where $S_\alpha$ was defined in \eqref{def:SA}. 
%Since $G^\pm_\ep$ can be written as $G^\pm_\ep=\varphi_{3}^\pm(\H,\L)$ with $$\varphi_{3}^\pm(\lambda_1,\lambda_2)=\frac{1}{\varphi_1(\lambda_1)\pm i\varphi_{2,\alpha}(\lambda_1,\lambda_2)}$$ by the joint spectral theory for $\H$ and $\mathcal L$, $G^\pm_\ep$ strongly commutes with $\H$ and $\L$. In particular, $G^\pm_\ep\in \mathbb B(H^{\ell}(\mathbb H^d)m
%\begin{align*}\|G^\pm_\ep\|\le \sigma_2^{-1},\quad \|\H\<\L\>^{-\alpha}G^\pm_\ep\|\le (2s\ep)^{-1},\quad \|\H^{1/2}\<\L\>^{-\alpha/2}G^\pm_\ep \<\L\>^{-\alpha/2}\H^{1/2}\|\lesssim \ep^{-1},\end{align*}where the first two bounds follow from \eqref{eq:bound_below_Hpm}, and the last one is obtained by interpolating between the second bound and its dual. The last bound and the assumption \ref{H2}, that is  $\H\ge C_1\L^s$, also imply \begin{equation}	\label{eq:SGS}	\|S_\alpha G^\pm_\ep S_\alpha\|\lesssim \ep^{-1}.\end{equation}
%
\begin{comment}
Since $\H G^\pm_\ep=1+(\sigma_1\pm i\sigma_2\pm i2s\ep \H\<\L\>^{-\alpha})G^\pm_\ep$, we further obtain
\begin{align}
\label{theorem_abstract_proof_1}
\|\H G^\pm_\ep\|\le 1+\|(\sigma_1\pm i\sigma_2\pm i 2s\ep \H\<\L\>^{-\alpha})G^\pm_\ep\|\lesssim 1+\<\sigma_1+\sigma_2\>\sigma_2^{-1}.
\end{align}
\end{comment}

{\it Step 3}. Let $a_+=\max\{a,0\}$, $\delta\ge0$ and recall the definition of $W_\ep$ from Lemma\til\ref{lemma_W}. Let us denote 
\begin{align*}
S_{\alpha,\delta}&:=S_\alpha\<\delta \L\>^{- (s-\alpha)_+/2}=\L^{s/2}\<\L\>^{-\alpha/2}\<\delta \L\>^{- (s-\alpha)_+/2},\\
F^\pm_\ep&:=W_\ep S_{\alpha,\delta} G^\pm_\ep S_{\alpha,\delta} W_\ep .
\end{align*}
Note that $S_{\alpha,\delta}\in \mathbb B(L^2(\mathbb H^d))$ if $\delta>0$, $S_{\alpha,0}=S_\alpha$ and thus $D(S_{\alpha,\delta})\subset H^s(\mathbb H^d)$ for all $\delta\ge0$, and $F^\pm_\ep\in\mathbb B(L^2(\mathbb H^d))$ for all $\sigma_1\in\R$, $\ep,\sigma_2>0$ and $\delta\ge0$.  We shall derive a uniform bound of $\|F^\pm_\ep\|$ with respect to $\sigma_1\in \R$, $\sigma_2>0$, $\ep>0$ and sufficiently small $\delta\ge0 $. To this end, we first show that $F^\pm_\ep$ satisfies
\begin{align}
\label{theorem_abstract_proof_2}
\pm i \partial_\ep F^\pm_\ep=I_1+I_2+I_3+I_4,
\end{align}
where $\partial_\ep F^\pm_\ep=\lim_{h\to0}h^{-1}(F^\pm_{\ep+h}-F^\pm_{\ep})$ and 
\begin{align*}
I_1&=iF^\pm_\ep A_\alpha-iA_\alpha F^\pm_\ep,\\
I_2&=\pm iW_\ep ' S_{\alpha,\delta} G^\pm_\ep   S_{\alpha,\delta} W_\ep \pm iW_\ep S_{\alpha,\delta} G^\pm_\ep   S_{\alpha,\delta} W_\ep ',\\
I_3&=-W_\ep D_\delta S_{\alpha,\delta} G^\pm_\ep S_{\alpha,\delta} W_\ep-W_\ep S_{\alpha,\delta} G^\pm_\ep S_{\alpha,\delta} D_\delta W_\ep,\\
I_4&=\mp i\ep W_\ep S_{\alpha,\delta} G^\pm_\ep [[\H,iA_\alpha],iA_\alpha] G^\pm_\ep S_{\alpha,\delta} W_\ep
\end{align*}
with
$$
D_\delta:= \psi(\L),
\qquad
\psi(\lambda):=
s\<\lambda\>^{-\alpha}\left(s-\alpha\lambda^2\<\lambda\>^{-2}-(s-\alpha)_+\delta^2\lambda^2\<\delta\lambda\>^{-2}\right). 
$$
By elementary calculus it is easy to see that, if $s\ge\alpha$, then $0<s\<\lambda\>^{-\alpha}\left(s-\alpha\lambda^2\<\lambda\>^{-2}\right) \le s^2$ 
and
$0\le s(s-\alpha)_+ \<\lambda\>^{-\alpha} (\delta\lambda)^2\<\delta\lambda\>^{-2} \le o(\delta)$ for $\delta\searrow0$,
so that $|\psi(\lambda)|\le s^2$ if $\delta$ is small enough.
If $s<\alpha$, then 
$$s^2 \ge s\<\lambda\>^{-\alpha}\left(s-\alpha\lambda^2\<\lambda\>^{-2}\right) \ge -2\left(\frac{\alpha-s}{\alpha+2}\right)^{1+\alpha/2} > -2s e^{-1-s/2}$$
and so $|\psi(\lambda)| \le s \max\{s,2e^{-1-s/2}\}$. Therefore by using \eqref{spectral_decomposition_theorem_2} we have that $D_\delta\in \mathbb B(L^2(\mathbb H^d))$ and 
\begin{align}
\label{theorem_abstract_proof_3}
\|D_\delta\| \le s\max\{s,2e^{-(1+s/2)}\}
\end{align}
for $\delta$ sufficiently small.
To prove \eqref{theorem_abstract_proof_2}, we first observe that the resolvent equation
\begin{align*}
G^\pm_{\ep+h}-G^\pm_{\ep}=\pm ih G^\pm_{\ep+h}[\H,iA_\alpha] G^\pm_\ep=\pm 2ish G^\pm_{\ep+h}\H\<\L\>^{-\alpha} G^\pm_\ep
%=\pm ish G^\pm_\ep_{\ep+h}Q_\alpha^2 G^\pm_\ep_\ep
\end{align*}
holds whenever $\ep>0$ and $\ep+h>0$, where the right hand side makes sense as a map from $\Ran S_{\alpha,\delta}$ to $D(S_{\alpha,\delta})$ thanks to the properties $G^\pm_\ep\in\mathbb B(H^{-s}(\mathbb H^d),H^{s}(\mathbb H^d))$ and $D(S_{\alpha,\delta}),D(\<\mathcal H\>^{1/2})\subset H^{s}(\mathbb H^d)$.  Hence, the map $(0,\infty)\ni \ep\mapsto F^\pm_\ep\in \mathbb B(L^2(\mathbb H^d))$ is $C^1$ and satisfies
\begin{align}
\label{theorem_abstract_proof_4}
\partial_\ep F^\pm_\ep&=\pm i W_\ep S_{\alpha,\delta} G^\pm_\ep    [\H,iA_\alpha]  G^\pm_\ep   S_{\alpha,\delta}  W_\ep + W_\ep ' S_{\alpha,\delta} G^\pm_\ep   S_{\alpha,\delta} W_\ep +W_\ep S_{\alpha,\delta} G^\pm_\ep   S_{\alpha,\delta}  W_\ep '\in \mathbb B(L^2(\mathbb H^d)).
\end{align}

We next compute the operator
$
J:=W_\ep S_{\alpha,\delta} [G^\pm_\ep,iA_\alpha]S_{\alpha,\delta}  W_\ep 
$ 
in the following two ways. On one hand, it holds that
\begin{align}
\label{theorem_abstract_proof_5}
W_\ep S_{\alpha,\delta} [G^\pm_\ep,iA_\alpha]S_{\alpha,\delta}  W_\ep =W_\ep S_{\alpha,\delta} G^\pm_\ep [iA_\alpha,\H^\pm_\ep] G^\pm_\ep S_{\alpha,\delta}  W_\ep \in \mathbb B(L^2(\mathbb H^d)). 
\end{align}
Indeed, setting for short $\H^\pm_\ep(\tau)=e^{-i\tau A_\alpha}\H^\pm_\ep e^{i\tau A_\alpha}$ and $G^\pm_\ep(\tau)=e^{-i\tau A_\alpha}G^\pm_\ep e^{i\tau A_\alpha}$, and recalling that $G^\pm_\ep=(\H^\pm_\ep)^{-1}$, we have
$$
G^\pm_\ep(\tau)=G^\pm_\ep-G^\pm_\ep(\tau)\{\H^\pm_\ep(\tau)-\H^\pm_\ep\}G^\pm_\ep,
$$
which implies
\begin{align}
\label{theorem_abstract_proof_5_0}
[G^\pm_\ep,iA_\alpha]=\frac{d}{d\tau}G^\pm_\ep(\tau)\Big|_{\tau=0}=-G^\pm_\ep\frac{d}{d\tau}\H^\pm_\ep(\tau)\Big|_{\tau=0}G^\pm_\ep=G^\pm_\ep [iA_\alpha,\H^\pm_\ep] G^\pm_\ep. 
\end{align}
By \eqref{def:Hep}, \eqref{theorem_abstract_proof_-2} and \eqref{theorem_abstract_proof_-1}, $[iA_\alpha,\H^\pm_\ep]$ is of the form
\begin{align}
\label{theorem_abstract_proof_5_1}
[iA_\alpha,\H^\pm_\ep]=-[\H,iA_\alpha]\pm i\ep [[\H,iA_\alpha],iA_\alpha]=-2s\H\<\L\>^{-\alpha}\pm 4is^2\ep \H\<\L\>^{-2\alpha}(1 -\alpha \L^2\<\L\>^{-2}).
\end{align}
In particular, $[iA_\alpha,\H^\pm_\ep]\in \mathbb B(H^{s}(\mathbb H^d),H^{-s}(\mathbb H^d))$ by \ref{H2}. 
These formulas \eqref{theorem_abstract_proof_5_0} and \eqref{theorem_abstract_proof_5_1}, combined with the same argument as that for $G_{\ep+h}^\pm-G^\pm_\ep$, %and the property $G^\pm_\ep:L^2\to H^{2s}(\mathbb H^d)=D(\H)$ show $G^\pm_\ep [iA_\alpha,\H^\pm_\ep] G^\pm_\ep\in \mathbb B(L^2(\mathbb H^d))$, and hence $G^\pm_\ep D(A_\alpha)\subset D(A_\alpha)$ by the same argument as in the proof of Lemma \ref{lemma_abstract_domain}. The formula \eqref{theorem_abstract_proof_5_1}, combined with \ref{H2} and \eqref{eq:SGS}, also implies 
thus show \eqref{theorem_abstract_proof_5}. 
%Thus, the same argument as in the proof of Lemma \ref{lemma_abstract_domain} shows $G^\pm_\ep D(A)\subset D(A)$ and $G^\pm_\ep D(A_\alpha)\subset D(A_\alpha)$. Moreover, \eqref{theorem_abstract_proof_5_0} implies  \eqref{theorem_abstract_proof_5} since $S_{\alpha,\delta}$ and $W_\ep$ are bounded on $L^2(\mathbb H^d)$. 
By \eqref{theorem_abstract_proof_5} and \eqref{theorem_abstract_proof_5_1}, we have
\begin{align*}
J %&=W_\ep S_{\alpha,\delta} [G^\pm_\ep,iA_\alpha]S_{\alpha,\delta}  W_\ep \\
&=W_\ep S_{\alpha,\delta} G^\pm_\ep[iA_\alpha,\H^\pm_\ep]G^\pm_\ep S_{\alpha,\delta}  W_\ep \\
&=-W_\ep S_{\alpha,\delta} G^\pm_\ep [\H,iA_\alpha] G^\pm_\ep S_{\alpha,\delta}  W_\ep \pm i\ep W_\ep S_{\alpha,\delta} G^\pm_\ep [[\H,iA_\alpha],iA_\alpha]G^\pm_\ep S_{\alpha,\delta}  W_\ep \\
&=\pm i \partial_\ep F^\pm_\ep \mp i\left(W_\ep 'S_{\alpha,\delta} G^\pm_\ep   S_{\alpha,\delta}  W_\ep +W_\ep S_{\alpha,\delta} G^\pm_\ep   S_{\alpha,\delta}  W_\ep '\right)\\
&\quad\quad \pm i\ep W_\ep S_{\alpha,\delta} G^\pm_\ep [[\H,iA_\alpha],iA_\alpha]G^\pm_\ep S_{\alpha,\delta}  W_\ep\\
&=\pm i \partial_\ep F^\pm_\ep-I_2-I_4.
\end{align*}
On the other hand, since Lemma \ref{lemma_abstract_commutator_1} implies
%
%\marginnote{Remark: only place in the proof where $L^{s/2}$ in the definition of $S_\alpha$ is used}
%
\begin{align*}
[S_{\alpha,\delta},iA_\alpha]
&=\<\L\>^{-\alpha/2}[S_{\alpha,\delta},iA]\<\L\>^{-\alpha/2}\\
&=2s\L \<\L\>^{-\alpha}\frac{d}{d\lambda}\left\{\lambda^{s/2}\<\lambda\>^{-\alpha/2}\<\delta\lambda\>^{-(s-\alpha)_+/2}\right\}\Big|_{\lambda=\L}\\
&=s S_{\alpha,\delta} \<\L\>^{-\alpha}\left(s-\alpha\L^2\<\L\>^{-2}-(s-\alpha)_+\delta^2\L^2\<\delta\L\>^{-2}\right)
\\
&=S_{\alpha,\delta} D_\delta= D_\delta S_{\alpha,\delta}
\end{align*}
with $D_\delta$ defined above, we compute $J$ as a quadratic form on $L^2(\mathbb H^d)$ as 
\begin{align*}
J&=W_\ep S_{\alpha,\delta} G^\pm_\ep iA_\alpha S_{\alpha,\delta} W_\ep - W_\ep S_{\alpha,\delta} iA_\alpha G^\pm_\ep S_{\alpha,\delta} W_\ep \\
&= W_\ep S_{\alpha,\delta} G^\pm_\ep S_{\alpha,\delta} iA_\alpha W_\ep  + W_\ep S_{\alpha,\delta} G^\pm_\ep[iA_\alpha,S_{\alpha,\delta}] W_\ep \\
&\quad\quad - W_\ep iA_\alpha S_{\alpha,\delta} G^\pm_\ep S_{\alpha,\delta} W_\ep - W_\ep [S_{\alpha,\delta}, iA_\alpha] G^\pm_\ep S_{\alpha,\delta} W_\ep \\
&=iF^\pm_\ep A_\alpha-iA_\alpha F^\pm_\ep-W_\ep S_{\alpha,\delta} G^\pm_\ep S_{\alpha,\delta} D_\delta W_\ep - W_\ep D_\delta S_{\alpha,\delta} G^\pm_\ep S_{\alpha,\delta} W_\ep\\
&=I_1+I_3.
\end{align*}
Now \eqref{theorem_abstract_proof_2} follows from these two formulas for $J$. 

{\it Step 4}. Using \eqref{theorem_abstract_proof_2}, we shall derive the following integral inequality: 
\begin{align}
\label{theorem_abstract_proof_6}
\|F^\pm_\ep\|\le \|F^\pm_b\|+\int_{\ep}^b\left(u_1 \|F^\pm_\tau\|+u_2(\tau)\|F^\pm_\tau\|^{1/2}\right)d\tau,
\end{align}
where $0<\ep<b<\infty$, $u_1=2C_1^{-1}C_2s$ and
$$
u_2(\tau)=2^{1/2}C_1^{-1/2} \max\{s,2e^{-(1+s/2)}\} s^{1/2} \tau^{-1/2}+2^{1/2}(2-\mu)(C_1s)^{-1/2}\tau^{\mu-3/2}. 
$$
For $f\in L^2(\mathbb H^d)$ and letting $g:=G^\pm_\ep S_{\alpha,\delta}  W_\ep f\in H^s(\mathbb H^d)$ for short , we know by the bounds $\|W_\ep \|\le1$ and $\|\<\delta \L\>^{-(s-\alpha)_+/2}\|\le1$, \ref{H2}, \eqref{def:Hep} and the fact $G^\pm_\ep=(\mathcal H^\pm_\ep)^{-1}$ that
\begin{align*}
\|F^\pm_\ep f\|^2
&= \|W_\ep S_{\alpha,\delta} G^\pm_\ep S_{\alpha,\delta}  W_\ep f\|^2\\
&\le \|S_{\alpha} g\|^2\\
%&\le \|Q_\alpha G^\pm_\ep S_{\alpha,\delta}  W_\ep f\|^2\\
&=\<\L^{s/2}\<\L\>^{-\alpha/2}g,\L^{s/2}\<\L\>^{-\alpha/2}g\>\\
&\le C_1^{-1}\<\H\<\L\>^{-\alpha/2}g,\<\L\>^{-\alpha/2}g\>\\
&=(2C_1s)^{-1}\<[\H,iA]g,g\>\\
&\le \mp (2C_1s\ep)^{-1} \Im \<\H^\pm_\ep g,g\>\\
&=\mp (2C_1s\ep)^{-1}\Im\<f,F^\pm_\ep f\>\\
&\le (2C_1s\ep)^{-1}\|F^\pm_\ep f\| \|f\|,
\end{align*}
 where we also used the condition $\sigma_2\ge0$ in the sixth line. In particular, 
\begin{align}
\label{theorem_abstract_proof_7}
\|F^\pm_\ep\|&\le (2C_1s\ep)^{-1},\\
\label{theorem_abstract_proof_8}
\|S_{\alpha,\delta} G^\pm_\ep S_{\alpha,\delta}W_\ep \|&\le (2C_1s\ep)^{-1/2}\|F^\pm_\ep\|^{1/2}. 
\end{align}
Moreover, Lemma \ref{lemma_W}, the bounds \eqref{theorem_abstract_proof_3} and \eqref{theorem_abstract_proof_8}, and the duality argument show
\begin{align*}
\|I_1\|&\le \|F^\pm_\ep A_\alpha\|+\|A_\alpha F^\pm_\ep\|\le 2\|W_\ep S_{\alpha,\delta}  G^\pm_\ep S_{\alpha,\delta} \| \| W_\ep A_\alpha\|\\
&\le 2^{1/2}(C_1s)^{-1/2}\ep^{\mu-3/2}\|F^\pm_\ep\|^{1/2},\\
\|I_2\|&\le \|W_\ep 'S_{\alpha,\delta} G^\pm_\ep   S_{\alpha,\delta}  W_\ep \|+\|W_\ep S_{\alpha,\delta} G^\pm_\ep   S_{\alpha,\delta}  W_\ep '\|\le 2\|W_\ep '\|\|S_{\alpha,\delta} G^\pm_\ep   S_{\alpha,\delta}  W_\ep \|\\
&\le 2^{1/2}(1-\mu)(C_1s)^{-1/2}\ep^{\mu-3/2}\|F^\pm_\ep\|^{1/2},\\
\|I_3\|&\le \|W_\ep S_{\alpha,\delta} G^\pm_\ep S_{\alpha,\delta} D_\delta W_\ep \|+\|W_\ep D_\delta S_{\alpha,\delta} G^\pm_\ep S_{\alpha,\delta} W_\ep \|\le 2\|W_\ep \|\| D_\delta \| \|S_{\alpha,\delta} G^\pm_\ep S_{\alpha,\delta} W_\ep \|\\
&\le 2^{1/2}C_1^{-1/2} \max\{s,2e^{-(1+s/2)}\} s^{1/2} \ep^{-1/2}\|F^\pm_\ep\|^{1/2}. 
%2\|W_\ep S_{\alpha,\delta} G^\pm_\ep S_{\alpha,\delta}\|\|D_\delta\|\|W_\ep \|
\end{align*}
We also know by \eqref{theorem_abstract_proof_0} and \eqref{theorem_abstract_proof_8} that
\begin{align*}
\|I_4\|&=\ep \|W_\ep S_{\alpha,\delta} G^\pm_\ep [[\H,iA_\alpha],iA_\alpha] G^\pm_\ep S_{\alpha,\delta} W_\ep \|\\
&\le 4C_2 s^2 \ep \|W_\ep S_{\alpha,\delta} G^\pm_\ep S_\alpha\|\|S_\alpha G^\pm_\ep S_{\alpha,\delta} W_\ep \|\\
&\le 2C_1^{-1}C_2 s \|F^\pm_\ep\|.
\end{align*}
Applying these four bounds for $I_1$,...,$I_4$ to \eqref{theorem_abstract_proof_2} imply 
\begin{align}
\label{theorem_abstract_proof_9}
\|\partial_\ep F^\pm_\ep\|\le u_1 \|F^\pm_\ep\|+u_2(\ep)\|F^\pm_\ep\|^{1/2}
\end{align}
and hence 
	from this and the fundamental theorem of calculus
	\begin{equation*}
		F^\pm_\ep = F^\pm_b - \int_\ep^b \partial_\tau F^\pm_\tau \, d\tau
	\end{equation*}
the inequality \eqref{theorem_abstract_proof_6} follows. Since $\mu>1/2$ and thus $u_2\in L^1(0,b)$ for any $b>0$, we can apply Lemma~\ref{lemma_Gronwall} to obtain
\begin{equation*}
\begin{aligned}
\|F^\pm_\ep\|
&\le \left\{ \|F^\pm_b\|^{1/2} + \frac{1}{2} \int_{\ep}^bu_2(\tau) e^{-\frac{u_1}{2}(b-\tau)} d\tau\right\}^2e^{(b-\ep)u_1}\\
&\le \left\{ (2C_1sb)^{-1/2} e^{\frac{u_1}{2}b}+ \frac{1}{2} \int_{0}^bu_2(\tau) e^{\frac{u_1}{2}\tau} d\tau\right\}^2\\
&=
\left\{
(2C_1sb)^{-1/2}
+
(2C_1^{-1}s)^{1/2} \max\{s,2e^{-(1+s/2)}\} \left(\frac{2}{u_1}\right)^{1/2} D\left(2, \left(\frac{u_1}{2}b \right)^{1/2} \right) 
\right.
\\&\qquad
+
\left.
2(2-\mu)(2C_1s)^{-1/2} \frac{1}{2\mu-1} \left(\frac{2}{u_1}\right)^{\mu-1/2} D\left(\frac{2}{2\mu-1} , \left(\frac{u_1}{2}b \right)^{\mu-1/2} \right)
\right\}^2
e^{u_1 b}
\end{aligned}
\end{equation*}
%	\begin{comment}
%		\begin{equation*}
%			\begin{aligned}
%				\|F^\pm_\ep\|
%				&\le \left\{ \|F^\pm_b\|^{1/2} + \frac{1}{2} \int_{\ep}^bu_2(\tau) e^{-\frac{u_1}{2}(b-\tau)} d\tau\right\}^2e^{(b-\ep)u_1}\\
%				&\le \left\{
%				\|F^\pm_b\|^{1/2}
%				+
%				(2C_1^{-1}sb)^{1/2} \max\{s,2e^{-(1+s/2)}\} + 2^{1/2} \frac{2-\mu}{2\mu-1}(C_1s)^{-1/2} b^{\mu-1/2}
%				\right\}^2
%				e^{
	%				bu_1}
				%\\&\le\frac{e^{C_2s^{-1}\max\{s,2e^{-(1+s/2)}\}(2C_1^{-1}sb)}}{(2C_1^{-1}sb)}\left\{C_1^{-1}+\max\{s,2e^{-(1+s/2)}\}(2C_1^{-1}sb)+\frac{2-\mu}{2\mu-1} 2^{1-\mu} C_1^{\mu-1} s^{-\mu}(2C_1^{-1}sb)^{\mu}\right\}^2
%				\\
%				&\le
%				\frac{e^{C_2\widetilde{b}}}{\widetilde{b}}
%				\left\{
%				C_1^{-1}
%				+
%				\max\{s,2e^{-(1+s/2)}\}
%				\widetilde{b}
%				+
%				\frac{2-\mu}{2\mu-1} \left(\frac{2}{C_1}\right)^{1-\mu} s^{-\mu}
%				\widetilde{b}^{\mu}
%				\right\}^2
%			\end{aligned}
%		\end{equation*}
%	\end{comment}
for any $0<\ep<b$, where $D(p,x)$ is the generalized Dawson's integral \eqref{def:dawson}, and we used \eqref{theorem_abstract_proof_7} with $\ep=b$ and a change of variables.
Setting $\widetilde{b}=\frac{u_1}{2}b$ and taking the infimum over $\widetilde{b}>0$, we have
\begin{align}
\label{theorem_abstract_proof_10}
\|F^\pm_\ep\|\le \kappa(\H,s,\mu),
\end{align}
for any $\sigma_1\in \R$, $\sigma_2>0$, $\ep>0$ and sufficiently small $\delta\ge0 $ with $\kappa(\H,s,\mu)$ given by \eqref{def:kappa}. 

{\it Step 5}. 
Plugging \eqref{theorem_abstract_proof_10} to \eqref{theorem_abstract_proof_9} implies $\|\partial_\ep F^\pm_\ep\|=O(\ep^{\mu-3/2})$. Since $\mu>1/2$ and thus $\ep^{\mu-3/2}\in L^1((0,1])$, the map $(0,\infty)\ni \ep\mapsto F^\pm_\ep\in \mathbb B(L^2(\mathbb H^d))$ is H\"older continuous of order $\mu-1/2$, so the limit $F_0^\pm:= \lim_{\ep\searrow0}F^\pm_\ep\in \mathbb B(L^2(\mathbb H^d))$ exists. Setting $G^\pm_0:=(\H-\sigma_1\mp i\sigma_2)^{-1}$, we have the resolvent equation 
$$
G^\pm_\ep-G^\pm_0=\pm is\ep G^\pm_\ep\<\L\>^{-\alpha} \H G^\pm_0
$$
and thus,
%
%\marginnote{\teal{H: An important point is $G^\pm_\ep\to G^\pm_0$ in norm as $\ep \to +0$ for each $\sigma_1,\sigma_2$, which is not straightforward since $\mathcal H^\pm_\ep$ is not self-adjoint.} \blue{N: I kept the computation and the dependence. I suppressed old equation (3.14) since unnecessary.}}
%
by the spectral theorem and \eqref{eq:G_ep}, 
\begin{align*}
	\|G^\pm_\ep-G^\pm_0\|\le s \ep\|G_\ep^\pm\|\|\<\L\>^{-\alpha}\|\|\H G_0^\pm\|\le s\ep \sigma_2^{-1}\sup_{\lambda\ge0}|\lambda(\lambda-\sigma_1\mp i\sigma_2)^{-1}| = s\ep \sigma_2^{-1} \< \sigma_1 \sigma_2^{-1} \>.
\end{align*}
This shows $F_0^\pm=
\<A_\alpha\>^{-\mu}S_{\alpha,\delta}G^\pm_0 S_{\alpha,\delta}\<A_\alpha\>^{-\mu}
$ for each $\sigma_1\in \R$ and $\sigma_2>0$. As $\ep\searrow0$ in \eqref{theorem_abstract_proof_10}, we thus have
\begin{align}
\label{theorem_abstract_proof_11}
\|\<A_\alpha\>^{-\mu}S_{\alpha,\delta}(\H-\sigma_1\mp i\sigma_2)^{-1}S_{\alpha,\delta}\<A_\alpha\>^{-\mu}\|\le \kappa(\H,s,\mu).
\end{align}
Finally, letting $\delta\searrow0$, we find for any $\sigma\in \C\setminus\R$, $$\<A_\alpha\>^{-\mu}S_{\alpha,\delta}(\H-\sigma)^{-1}S_{\alpha,\delta}\<A_\alpha\>^{-\mu}\to \<A_\alpha\>^{-\mu}S_\alpha(\H-\sigma)^{-1}S_\alpha\<A_\alpha\>^{-\mu}$$ strongly on $L^2(\mathbb H^d)$ since $\<\delta \L\>^{- (s-\alpha)_+/2}\to 1$ strongly and $\|S_\alpha(\H-\sigma)^{-1}S_\alpha\|\lesssim \<\sigma\>|\Im\sigma|^{-1}$ by \ref{H2} and the spectral theorem.  The desired estimate \eqref{theorem_abstract_1} thus follows by letting $\delta\searrow 0$ in \eqref{theorem_abstract_proof_11}. 
\end{proof}

\section{Proof of the main theorems}\label{sec:proof_thm}
This section is devoted to the proof of Theorems \ref{theorem_resolvent_general} and \ref{theorem_resolvent_radial}. We first prepare a few lemmas. 

%{lemma}
\begin{lemma}[Hardy's type inequalities]
\label{lemma_Hardy}
The following statements hold
\begin{itemize}
\item For $d\ge1$ and $f\in H^1(\mathbb H^d)$,
\begin{align}
\label{lemma_Hardy_1}
\||(z,t)|_{\mathbb H}^{-2}|z|f\|\le d^{-1}\|\L^{1/2}f\|.
\end{align}
\item For $d\ge2$ and $f\in H^1(\mathbb H^d)$,
\begin{align}
\label{lemma_Hardy_2}
\||z|^{-1}f\|\le (d-1)^{-1}\|\L^{1/2}f\|.
\end{align}
%\item For $d\ge3$ and $f\in H^2(\mathbb H^d)$, \begin{align}\label{lemma_Hardy_3}\| |z|^{-2}f\|\le \frac{1}{d(d-2)}\|\L f\|.\end{align}\item For $0<s<d+1$, there exists $C_{s,d}>0$ such that  for $f\in H^{s}(\mathbb H^d)$, \begin{align}\label{lemma_Hardy_4}\| |(z,t)_{\mathbb H}|^{-s} f\|\le C_{d,s}\|\L^{s/2} f\|.\end{align}
\end{itemize}
\end{lemma}

%proof
\begin{proof}
See \cite[Corollary 2.1]{GL} for \eqref{lemma_Hardy_1} and \cite[Theorems 2.2 and 3.6]{DAm} for \eqref{lemma_Hardy_2}, respectively. % and \eqref{lemma_Hardy_3},  and \cite[Theorem A]{CCR}  for \eqref{lemma_Hardy_4}, respectively. 
\end{proof}

%remark
\begin{remark}
The constants in \eqref{lemma_Hardy_1} and \eqref{lemma_Hardy_2} are known to be sharp. 
\end{remark}

%{lemma}
\begin{lemma}
\label{lemma_XYT}
For $d\ge1$ and $f\in H^{1}(\mathbb H^d)$ and $g\in D(N|z|^{-1})\cap H^1(\mathbb H^d)$, 
\begin{align}
\label{lemma_XYT_1}
\|\nabla_{\mathbb H}f\|\le \|\L^{1/2}f\|,\quad \||z|Tg\|\le \frac12\|\L^{1/2} g\|+\frac12\|N|z|^{-1}g\|.
\end{align}
In particular, if $d\ge2$ and $f\in H^{1}(\mathbb H^d)$ then
\begin{align}
\label{lemma_XYT_2}
\|\<N\>^{-1}|z|Tf\|\le \frac{d}{2(d-1)}\|\L^{1/2} f\|. 
\end{align}
Moreover, for $d\ge1$ and $f\in H^2(\mathbb H^d)$, 
\begin{align}
\label{lemma_XYT_3}
\|Tf\|\le \frac 12 \|\L f\|.
\end{align}
%Moreover, if $f$ is radial, i.e. $f(z,t)=f(|z|,t)$ then $\||z|Tf\|\le \frac12\|\L^{1/2}f\|$. 
\end{lemma}

%remark
\begin{remark}
For $f\in \Ker N\cap H^1(\mathbb H^d)$, the last estimate in \eqref{lemma_XYT_1} particularly implies 
\begin{align}
\label{lemma_XYT_4}
\||z|Tf\|\le \frac12\|\L^{1/2} f\|
\end{align}
since $N$ commutes with the multiplication by any cylindrical function in $z$. 
\end{remark}

%proof
\begin{proof}[Proof of Lemma \ref{lemma_XYT}]
%The lemma is well known (see e.g. \cite[Section 4.4.4]{FR}). We however give an elementary proof for the sake of self-containedness. The first two estimates are trivial since
Since $\L=-\nabla_{\mathbb H}\cdot\nabla_{\mathbb H}$, the first estimate in \eqref{lemma_XYT_1} follows. To show the second estimate in \eqref{lemma_XYT_1}, we recall $N=z^\perp\cdot \partial_z$ and \eqref{sublaplacian_2}. These formulas imply the identity
$$
2|z|T=2|z|^{-1}z^\perp \cdot z^\perp T=|z|^{-1}z^\perp\cdot \nabla_{\mathbb H}-N|z|^{-1}
$$
and the second estimate in \eqref{lemma_XYT_1} thus follows, where we used the fact that $|z|^{-1}$ commutes with $N$. For $d\ge2$, \eqref{lemma_Hardy_2} and the trivial bound $\|\<N\>^{-1}|N|^{k}\|\le1$ for $k=0,1$ imply \eqref{lemma_XYT_2}. Finally, it follows from the commutation relation $T=\frac14(Y_1X_1-X_1Y_1)$ in \eqref{eq:commutators} and from \eqref{lemma_XYT_1} that for any $f,g\in H^1(\mathbb H^d)$
\begin{align*}
|\<Tf,g\>|\le \frac14(\|X_1f\|\|Y_1g\|+\|Y_1f\|\|X_1g\|)\le \frac12\|\L^{1/2}f\|\|\L^{1/2}g\|.
\end{align*}
%\marginnote{I plugged $+i\ep$ instead of $+\ep$ to use one of these formulas when I cite Schm\"udgen's theorem.}	
Plugging $f=(\L+i\ep)^{-1/2}\tilde f$ and $g=(\L+i\ep)^{-1/2}\tilde g$ shows 
$$
|\<(\L+i\ep)^{-1/2}T(\L+i\ep)^{-1/2} \tilde f,\tilde g\>|\le \frac12\|\tilde f\|\|\tilde g\|
$$
for any $\ep>0$ and $\tilde f,\tilde g \in L^2(\mathbb H^d)$. Since $\L$ commutes strongly with $T$ (see Appendix \ref{sec:Ls} below), the same argument as in the Step 1 of the proof of Theorem \ref{theorem_abstract} implies that $(\L+i\ep)^{-1/2}$ also commutes with $T$. Hence, we obtain 
$$
|\<T(\L+i\ep)^{-1}\tilde f,\tilde g\>|\le \frac12\|\tilde f\|\|\tilde g\|,
$$
which implies
\begin{equation}\label{eq:schmudgen_condition_T}
	\|Tf\|\le \frac12\|(\L+i\ep)f\|
\end{equation}
for any $\ep>0$ and $f\in H^2(\mathbb H^d)$, and \eqref{lemma_XYT_3} follows by letting $\ep\searrow0$. 
\end{proof}

%{lemma}
\begin{lemma}
\label{lemma_N}
The operator $\varphi(\L)$ leaves $\Ker N$ invariant: $\varphi(\L)\Ker N\subset \Ker N$ for any $\varphi\in L^\infty([0,\infty))$.
\end{lemma}

%proof
\begin{proof}
We first observe $N$ commutes with $\L$. Indeed, each $L_{d+j,j}=y_j\partial_{x_j}-x_j\partial_{y_j}$ commute with $\partial_z^2$ since $\partial_z^2$ is rotationally invariant, namely $R_{jk}(\theta)^{-1}\partial_z^2R_{jk}(\theta)=\partial_z^2$ for all $\theta\in \R$. Thus, $N=\sum_{j=1}^dL_{d+j,j}$ commutes with $\partial_z^2$. Since $N$ clearly also commutes with $|z|^2T^2$ and $NT$, so with $\L=\partial_z^2+4NT+4|z|^2T^2$. The same argument as in the proof of Theorem\til\ref{theorem_abstract} then yields that $N$ also commutes with $\varphi(\L)$ and the lemma follows. 
\end{proof}

We here record a version of Hadamard's three line theorem. Let us set $$S=\{\zeta=a+ib\ |\ 0<a<1,\ b\in \R\}.$$ 

%lemma
\begin{lemma}
\label{lemma_interpolation}
Let $F(\zeta)$ be continuous and satisfy $|F(\zeta)|\lesssim \<r\>$ on $\overline S$, holomorphic in $S$ and 
$$
\sup_{b\in \R}|F(ib)|\le M_0,\quad \sup_{b\in \R}(1+|b|)^{-1}|F(1+ib)|\le M_1, 
$$
with some $M_0,M_1>0$. 
Then $|F(\zeta)|\le  |1+\sqrt2\zeta| M_0^{1-a}M_1^a$ on $S$. 
\end{lemma}

\begin{proof}
Let $\widetilde F(\zeta)=(1+\sqrt2\zeta)^{-1}F(\zeta)$. Clearly, $\widetilde F(\zeta)$ is bounded continuous on $\overline S$ and holomorphic in $S$. Moreover, 
$
|\widetilde F(ib)|\le M_0$ and $|\widetilde F(1+ib)|\le M_1
$ since $|1+\sqrt2(1+ib)|\ge 1+|b|$. 
Hence, classical Hadamard's three line theorem (see e.g. \cite[Appendix to IX.4]{ReSi2}) can be applied to obtain $|\widetilde F(\zeta)|\le M_0^{1-a}M_1^a$ on $S$. 
\end{proof}

The following proposition plays an important role in the proof of Theorems \ref{theorem_resolvent_general} and \ref{theorem_resolvent_radial}. 

%proposition
\begin{proposition}
\label{proposition_replace}
Let $0<a\le1$. Then the following estimates hold:
\begin{itemize}
\item For $d\ge1$, 
\begin{align}
\label{proposition_replace_1}
\|\<A_1\>\<\L\>^{1/2}\L^{-1/2}w_1\|&\le 5+3d^{-1},\\
\label{proposition_replace_2}
\|\<A_1\>^a\<\L\>^{a/2}\L^{-a/2}w_1^a\|&\le (1+\sqrt 2a)(5+3d^{-1})^a.
\end{align}
\item For $d\ge2$, 
\begin{align}
\label{proposition_replace_3}
\|\<A_1\>\<\L\>^{1/2}\L^{-1/2}w_2\|&\le 5+4(d-1)^{-1},\\
\label{proposition_replace_4}
\|\<A_1\>^a\<\L\>^{a/2}\L^{-a/2}w_2^a\|&\le (1+\sqrt 2a)(5+4(d-1)^{-1})^a,\\
\label{proposition_replace_5}
\|\<A_{1/2}\>^aw_3^a\<N\>^{-a}\|&\le \left(4+ d^{-1}+(d-1)^{-1}\right)^a,\\
\label{proposition_replace_6}
\|\<A\>\L^{-1/2}w_4\<N\>^{-1}\|&\le 3+(d-1)^{-1},\\
\label{proposition_replace_7}
\|\<A\>^a\L^{-a/2}w_4^a\<N\>^{-a}\|&\le (1+\sqrt2 a)(3+(d-1)^{-1})^a. 
\end{align}
%\item For $d\ge3$ and $0<a\le2$\begin{align}\label{proposition_replace_8}\|\<A_{1}\>^{a/2}w_5^a\|&\le C_8(d,a),\\\label{proposition_replace_9}\|\<A\>\L^{-1}w_6^2\|&\le C_9(d),\\\label{proposition_replace_10}\|\<A\>^{a/2}\L^{-a/2}w_6^a\|&\le (1+\sqrt2 a)C_9(d)^a. \end{align}\item For $0<a\le (d+1)/2$
\end{itemize}
\end{proposition}

%proof
\begin{proof}
To prove \eqref{proposition_replace_1} and \eqref{proposition_replace_2}, we first show
\begin{align}
\label{proposition_replace_1'}
\|w_1(\L+\ep)^{-1/2}\<\L\>^{1/2}\<A_1\>\|&\le 5+3d^{-1},\\
\label{proposition_replace_2'}
\|w_1^a(\L+\ep)^{-a/2}\<\L\>^{a/2}\<A_1\>^a\|&\le (1+\sqrt 2a)(5+3d^{-1})^a
\end{align}
for all $\ep>0$. To this end, we shall show
\begin{align}
\label{proposition_replace_proof_1}
\|w_1^{1+ib}(\L+\ep)^{-(1+ib)/2}\<\L\>^{(1+ib)/2}\<A_1\>\|\le 5+(3+2|b|)d^{-1}
\end{align}
for all $b\in \R$. We let $M_1=M_1(\L)=(\L+\ep)^{-(1+ib)/2}\<\L\>^{ib/2}$ and use Lemma \ref{lemma_abstract_commutator_1} to compute
\begin{align*}
[M_1,iA]&=2\L M_1'(\L)=M_1\{-(1+ib)\L(\L+\ep)^{-1}+ib \L^2\<\L\>^{-2}\}\\
&=M_1\{-1+(1+ib)\ep(\L+\ep)^{-1}-ib \<\L\>^{-2}\}=-M_1+M_2
\end{align*}
with $M_2=M_1\{(1+ib)\ep(\L+\ep)^{-1}-ib \<\L\>^{-2}\}$, and , recalling the definition \eqref{def:SA} of $A_\alpha$,
\begin{align*}
&(\L+\ep)^{-(1+ib)/2}\<\L\>^{(1+ib)/2}(iA_1+1)%=M_1\<\L\>^{1/2}(\<\L\>^{-1/2}iA\<\L\>^{-1/2}+1)
=M_1iA\<\L\>^{-1/2}+M_1\<\L\>^{1/2}\\
&=iAM_1\<\L\>^{-1/2}+[M_1,iA]\<\L\>^{-1/2}+M_1\<\L\>^{1/2}\\
%&=(iA-1)M_1\<\L\>^{-1/2}+M_2\<\L\>^{-1/2}+M_1\<\L\>^{1/2}\\
&=\{z\cdot\nabla_{\mathbb H}+2tT+d)\}M_1\<\L\>^{-1/2}+M_2\<\L\>^{-1/2}+M_1\<\L\>^{1/2}.
\end{align*}
Since $w_1|z|\le1$, $w_1|t|\le1$, we obtain by using Lemma \ref{lemma_XYT} that
\begin{align}
\label{proposition_replace_proof_2}
\|w_1z\cdot\nabla_{\mathbb H}M_1\<\L\>^{-1/2}\|&\le \|w_1z\|_{L^\infty(\mathbb{H}^d)} \|\L^{1/2}(\L+\ep)^{-1/2}\<\L\>^{-1/2}\|\le1,\\
\label{proposition_replace_proof_3}
2\|w_1tTM_1\<\L\>^{-1/2}\|&\le \|w_1t\|_{L^\infty(\mathbb{H}^d)} \| \L (\L+\ep)^{-1/2}\<\L\>^{-1/2}\|\le1.
\end{align}
We also know, by \eqref{lemma_Hardy_1} and $w_1\le w_4$, that
\begin{align}
\label{proposition_replace_proof_4}
d\|w_1M_1\<\L\>^{-1/2}\|\le \|\L^{1/2}(\L+\ep)^{-1/2}\<\L\>^{-1/2}\|\le 1
\end{align}
and, since $\|(\L+\ep)^{-1}\|\le \ep^{-1}$ by the spectral theorem, 
\begin{equation}
\label{proposition_replace_proof_5}
\begin{split}
	\|w_1M_2\<\L\>^{-1/2}\|
	&\le d^{-1}(1+|b|)\|\L^{1/2}\ep(\L+\ep)^{-3/2}\<\L\>^{-1/2}\|+d^{-1}|b|\|\L^{1/2}(\L+\ep)^{-1/2}\<\L\>^{-2}\|\\
	&\le (1+2|b|)d^{-1}
\end{split}
\end{equation}
where we used also \eqref{proposition_replace_proof_4}.
To deal with the term $M_1\<\L\>^{1/2}$, recalling that $\supp E_\L=\sigma(\L)=[0,\infty)$, we decompose the identity as 
$\Id=E_\L([0,\delta])+E_\L([\delta,\infty))$ for a fixed $\delta>0$, and apply \eqref{lemma_Hardy_1} to the part $E_\L([0,\delta])$ and the bound $|w_1|\le1$ to $E_\L([\delta,\infty))$, respectively, obtaining
\begin{align}
\|w_1M_1\<\L\>^{1/2}\|
&\le d^{-1}\|\L^{1/2}(\L+\ep)^{-1/2}\<\L\>^{1/2}E_\L([0,\delta])\|+\|(\L+\ep)^{-1/2}\<\L\>^{1/2}E_\L([\delta,\infty))\|\\
&\le (d^{-1}+\delta^{-1/2}) \jap{\delta}^{1/2}
%2d^{-1}+2.
\end{align}
and hence, taking the infimum for $\delta>0$,
\begin{equation}
	\label{proposition_replace_proof_6}
	\|w_1M_1\<\L\>^{1/2}\| \le (1+d^{-4/5})^{5/4} < 2(1+d^{-1}).
\end{equation}
%
%\marginnote{I minimized using $\delta$ but in the end I kept $2(1+1/d)$}
The estimates \eqref{proposition_replace_proof_2}--\eqref{proposition_replace_proof_6} and the simple bound $\|(iA_1+1)^{-1}\<A_1\>\|\le1$ imply \eqref{proposition_replace_proof_1} and \eqref{proposition_replace_1'}. 

For proving \eqref{proposition_replace_2'}, we consider a complex function
$$
\varphi(\zeta)=\<w_1^{\zeta}(\L+\ep)^{-\zeta/2}\<\L\>^{\zeta/2}\<A_1\>^\zeta f,g\>
$$
for $f,g\in \S(\mathbb H^d)$ with $\|f\|=\|g\|=1$ and $0\le \Re \zeta\le1$. Since 
$$
|\varphi(\zeta)|\lesssim \ep^{-1/2}\|\<A_1\>^{|\Re \zeta|}f\|\|g\|\lesssim \ep^{-1/2}\|\<A_1\>f\|\|g\|<\infty
$$
for each $\ep>0$ and $0\le\Re \zeta\le1$, $\varphi(\zeta)$ is bounded continuous on the strip $\{\zeta\in\C\ |\ 0\le \Re \zeta\le1\}$. Moreover, $\varphi(\zeta)$ is holomorphic in the interior $\{\zeta\in\C\ |\ 0<\Re \zeta<1\}$ and \eqref{proposition_replace_proof_1} implies
$$
|\varphi(ib)|\le 1,\quad |\varphi(1+ib)|\le 5+(3+2|b|)d^{-1}\le (5+3d^{-1})(1+|b|). 
$$
Lemma \ref{lemma_interpolation} then implies 
$
|\varphi(a)|\le (1+\sqrt 2a)(5+3d^{-1})^a
$ for $0<a<1$
and \eqref{proposition_replace_2'} follows. 

Now we observe that, for any $f\in \S(\mathbb H^d)$ and $0\le a\le 1$, $w_1^a(\L+\ep)^{-a/2}\<\L\>^{a/2}\<A_1\>^af$ converges to $w_1^a\L^{-a/2}\<\L\>^{a/2}\<A_1\>^af$ strongly in $L^2(\mathbb H^d)$ as $\ep\searrow0$. Indeed, we have $g:=\<A_1\>^af\in L^2(\mathbb H^d)$ and, thanks to the same argument as in \eqref{proposition_replace_proof_6} and the spectral decomposition theorem, 
\begin{align*}
&\|w_1^a\{(\L+\ep)^{-a/2}-\L^{-a/2}\}\<\L\>^{a/2}g\|^2\\
&\lesssim \|E_\L([0,1])\L^{a/2}\{(\L+\ep)^{-a/2}-\L^{-a/2}\}\<\L\>^{a/2}g\|^2+\|E_\L([1,\infty])\{(\L+\ep)^{-a/2}-\L^{-a/2}\}\<\L\>^{a/2}g\|^2\\
&\lesssim \int_0^1\<\lambda\>^{a}\left|\left(\frac{\lambda}{\lambda+\ep}\right)^{a/2}-1\right|^2d\<E_{\L}(\lambda)g,g\>+\int_1^\infty \frac{\<\lambda\>^{a}}{\lambda^a}\left|\left(\frac{\lambda}{\lambda+\ep}\right)^{a/2}-1\right|^2d\<E_{\L}(\lambda)g,g\>
\to 0
\end{align*}
as $\ep\searrow0$ by the dominated convergence theorem. %, since $\left\lvert \left(\frac{\lambda}{\lambda+\ep}\right)^{a/2}-1 \right\rvert^2 \le O(\lambda^{-2})$. 
Hence, we obtain \eqref{proposition_replace_1} and \eqref{proposition_replace_2} by letting $\ep\searrow0$ and taking the adjoints in \eqref{proposition_replace_1'} and \eqref{proposition_replace_2'}. 

The proof of \eqref{proposition_replace_3} and \eqref{proposition_replace_4} is almost identical. Precisely, it is enough to show 
\begin{align}
\label{proposition_replace_proof_7}
\|w_2^{1+ib}(\L+\ep)^{-(1+ib)/2}\<\L\>^{(1+ib)/2}(iA_1+1)\|\le 5+(4+2|b|)(d-1)^{-1}
\end{align}
for all $\ep>0$ and $b\in \R$. To this end, by the same argument as above, we observe that
\begin{align*}
\|w_2z\cdot\nabla_{\mathbb H}M_1\<\L\>^{-1/2}\|&\le1,\\
2\|w_2tTM_1\<\L\>^{-1/2}\<\L\>^{-1/2}\|&\le 1. 
\end{align*}
Moreover, using \eqref{lemma_Hardy_2} instead of \eqref{lemma_Hardy_1} (for which we need the assumption $d\ge2$) imply
\begin{align*}
d\|w_2M_1\<\L\>^{-1/2}\|&\le 1+(d-1)^{-1},\\
\|w_2M_2\<\L\>^{-1/2}\|&\le (1+2|b|)(d-1)^{-1},\\
\|w_2M_1\<\L\>^{1/2}\|&\le 2+2(d-1)^{-1}
\end{align*}
and \eqref{proposition_replace_proof_7} thus follows from these five bounds. 

To prove \eqref{proposition_replace_5},  we compute by using Lemma \ref{lemma_abstract_commutator_1} that
\begin{align*}
iA_{1/2}+1&=iA\<\L\>^{-1/2}+[\<\L\>^{-1/4},iA]\<\L\>^{-1/4}+1\\
&=(iA-2^{-1})\<\L\>^{-1/2}+2^{-1}\<\L\>^{-5/2}+1
\end{align*}
and use the same argument as above and \eqref{lemma_Hardy_1} to obtain
\begin{align}
\nonumber
&\|\<N\>^{-1}w_3(z\cdot\nabla_{\mathbb H}+d+2^{-1})\<\L\>^{-1/2}\|+2^{-1}\|\<N\>^{-1}w_3\<\L\>^{-5/2}\|+\|\<N\>^{-1}w_3\|\\
%&\|w_3(z\cdot\nabla_{\mathbb H}+d+2^{-1})\<\L\>^{-1/2}\|+2^{-1}\|w_3\<\L\>^{-5/2}\|+\|w_3\|\\
\label{proposition_replace_proof_8}
&\le\{1+(d+2^{-1})d^{-1}\}\|\L^{1/2}\<\L\>^{-1/2}\|+(2d)^{-1}\|\L^{1/2}\<\L\>^{-5/2}\|+1
\le 3+d^{-1}. 
\end{align}
To deal with the term associated with $2tT$, taking  into account the fact that $\<N\>$ commutes with any cylindrical function,
and using $w_3 |t| \le |z|$ and \eqref{lemma_XYT_2}, we obtain
\begin{align*}
%\label{proposition_replace_proof_9}
\|\<N\>^{-1-ib}w_3^{1+ib}2tT\<\L\>^{-1/2}\|\le
2 \|\<N\>^{-1} |z| T\<\L\>^{-1/2}\|\le
 \frac d{d-1} \|\L^{1/2}\<\L\>^{-1/2}\|\le 1+(d-1)^{-1},
\end{align*}
where we used the assumption $d\ge2$ to apply \eqref{lemma_XYT_2}. Hence, we have
\begin{align*}
\|\<N\>^{-1-ib}w_3^{1+ib}\<A_{1/2}\>^{1+ib}\|%\le 2+d^{-1}+d^{-1}/2+1+(d-1)^{-1}+1
\le 4+ d^{-1}+(d-1)^{-1}.
\end{align*}
Since the right hand side is bounded with respect to $b$, classical Hadamard's three line theorem shows
\begin{align*}
\|\<N\>^{-a}w_3^a\<A_{1/2}\>^{a}\|%\le 2+d^{-1}+d^{-1}/2+1+(d-1)^{-1}+1
\le \left(4+ d^{-1}+(d-1)^{-1}\right)^a
\end{align*}
for $0\le a\le 1$ and \eqref{proposition_replace_5} follows by the duality.  

Similarly, \eqref{proposition_replace_6} and \eqref{proposition_replace_7} also follow from the complex interpolation and the estimate
$$
\|\<N\>^{-1}w_4(\L+\ep)^{-(1+ib)/2}(iA-1)\|\le 3+(d-1)^{-1}+2|b|d^{-1}. %\le (3+(d-1)^{-1})(1+|b|). 
$$
As above, this estimate is obtained by combining with the formula
\begin{align*}
(\L+\ep)^{-(1+ib)/2}(iA-1)%\\&=(iA-1)(\L+\ep)^{-(1+ib)/2}+[(\L+\ep)^{-(1+ib)/2},iA]\\&=(iA-1)(\L+\ep)^{-(1+ib)/2}+2\L\left\{-(1+ib)/2\right\}(\L+\ep)^{-(3+ib)/2}\\&=(iA-1)(\L+\ep)^{-(1+ib)/2}-(1+ib)\L(\L+\ep)^{-(3+ib)/2}\\&=(iA-1)(\L+\ep)^{-(1+ib)/2}-(1+ib)(\L+\ep)^{-(1+ib)/2}+(1+ib)\ep(\L+\ep)^{-(3+ib)/2}
=(z\cdot\nabla_{\mathbb H}+2tT+d-1-ib)(\L+\ep)^{-(1+ib)/2}+(1+ib)\ep(\L+\ep)^{-(3+ib)/2}
\end{align*}
and the following two estimates: 
\begin{align}
\nonumber
&\|\<N\>^{-1}w_4(z\cdot\nabla_{\mathbb H}+d-1-ib)(\L+\ep)^{-(1+ib)/2}\|+|1+ib|\|\<N\>^{-1}w_4\ep(\L+\ep)^{-(3+ib)/2}\|\\
\nonumber
&\le \{1+(d-1+|b|)d^{-1}\}\|\L^{1/2}(\L+\ep)^{-1/2}\|+(1+|b|)d^{-1}\|\L^{1/2}(\L+\ep)^{-1/2}\|\|\ep(\L+\ep)^{-1}\|\\
\label{proposition_replace_proof_9}
&\le 2+2|b|d^{-1}
\end{align}
and 
$
\|\<N\>^{-1}w_42tT(\L+\ep)^{-1/2}\|\le 1+(d-1)^{-1}
$. 
\end{proof}

We are finally in a position to prove Theorem \ref{theorem_resolvent_general}.
%proof
\begin{proof}[Proof of Theorem \ref{theorem_resolvent_general}]Theorem \ref{theorem_resolvent_general} is a direct consequence of Theorem \ref{theorem_abstract} and Proposition \ref{proposition_replace}. Indeed, for the case $G=w_1^s\<\L\>^{(s-1)/2} $, we rewrite $\<\L\>^{(s-1)/2} w_1^sf$ with $f\in \S(\mathbb H^d)$ as 
$$
\<\L\>^{(s-1)/2} w_1^s f=\L^{s/2}\<\L\>^{-1/2}\<A_1\>^{-s}\cdot \<A_1\>^s\<\L\>^{s/2}\L^{-s/2}w_1^sf. 
$$
Theorem \ref{theorem_abstract} with $(\H,\mu,\alpha)=(\L^s,s,1)$ or $(\H,\mu,\alpha)=(\L_s,s,1)$, \eqref{proposition_replace_1} and  \eqref{proposition_replace_2} with $a=s\in (\frac12,1]$ yield 
\begin{align*}
|\<w_1^s\<\L\>^{(s-1)/2} (\H-\sigma)^{-1}\<\L\>^{(s-1)/2} w_1^sf,g\>|&\le \kappa(\H,s,s)\|\<A_1\>^s\<\L\>^{s/2}\L^{-s/2}w_1^sf\|\|\<A_1\>^s\<\L\>^{s/2}\L^{-s/2}w_1^sg\|\\
&\le C(s)\kappa(\H,s,s)\|f\|\|g\|
\end{align*}
for $g\in \S(\mathbb H^d)$, where $C(1)=(5+3d^{-1})^2$ and $C(s)=(1+\sqrt 2s)^2(5+3d^{-1})^{2s}$ for $1/2<s<1$. By taking the supremum over $g\in \S(\mathbb H^d)$ with $\|g\|=1$ and using the density argument, we arrive at
$$
\|w_1^s\<\L\>^{(s-1)/2} (\H-\sigma)^{-1}\<\L\>^{(s-1)/2}w_1^s f\|\le C(s)\kappa(\H,s,s)\|f\|.
$$
Similarly, the other three cases also follow from the same argument by writing
\begin{align*}
\<\L\>^{(s-1)/2} w_2^s f&=\L^{s/2}\<\L\>^{-1/2}\<A_1\>^{-s}\cdot \<A_1\>^s\<\L\>^{s/2}\L^{-s/2}w_2^sf,\\
\L^{s/2}\<\L\>^{-1/4}w_3^\mu \<N\>^{-\mu}f&=\L^{s/2}\<\L\>^{-1/4}\<A_{1/2}\>^{-\mu}\cdot \<A_{1/2}\>^{\mu}w_3^\mu \<N\>^{-\mu}f,\\
\L^{(s-\mu)/2} w_4^\mu \<N\>^{-\mu}f&=\L^{s/2}\<A\>^{-\mu}\cdot \<A\>^{\mu}\L^{-\mu/2}w_4^\mu \<N\>^{-\mu}f,
\end{align*}
and using Theorem \ref{theorem_abstract} and Proposition \ref{proposition_replace} with $a=\mu\in (\frac12,1]$. 
\end{proof}

%proof
\begin{proof}[Proof of Theorem \ref{theorem_resolvent_radial}]
Thanks to Theorem \ref{theorem_abstract}, Lemma \ref{lemma_N} and the same argument as above, Theorem \ref{theorem_resolvent_radial} follows from the following two estimates: 
\begin{align*}
\|\<A_{1/2}\>^\mu w_3^\mu f\|\le (4+d^{-1})^\mu \|f\|,\quad \|\<A\>^\mu \L^{-\mu/2}w_4^\mu f\|\le (1+\sqrt2 a)(3+2d^{-1})^\mu \|f\|
\end{align*}
for $f\in \S(\mathbb H^d)\cap \Ker N$. To this end, thanks to the same argument as above and the duality argument, it is enough to show
\begin{align}
\label{theorem_resolvent_radial_proof_1}
\|P_N w_3(iA_{1/2}-1)\|\le 4+d^{-1},\quad 
\|P_N w_4(\L+\ep)^{-(1+ib)/2} (iA-1)\|\le 3+2|b|d^{-1},
\end{align}
where $P_N$ denotes the projection onto $\Ker N$. The proof of these estimates is essentially the same as that of Lemma\til\ref{lemma_N}. The only difference is using \eqref{lemma_XYT_4} instead of \eqref{lemma_XYT_2} or \eqref{lemma_XYT_3} to deal with the terms associated with $2tT$. Precisely, we use  \eqref{lemma_XYT_4} to obtain, for all $d\ge1$ and $j=3,4$,
\begin{align*}
\|P_Nw_{j}(2tT)(\L+\ep)^{-(1+ib)/2}\|
\le 2\||z|TP_N(\L+\ep)^{-(1+ib)/2}\|
\le \|\L^{1/2}(\L+\ep)^{-(1+ib)/2}\|\le 1,
\end{align*}
where we also used the fact that $w_jtT$  commutes with $P_N$. This estimate together with \eqref{proposition_replace_proof_8} and \eqref{proposition_replace_proof_9} imply \eqref{theorem_resolvent_radial_proof_1}. 
\end{proof}

%appendix
\appendix
	
	\section{The fractional powers of the sublaplacian, and their spectral theory}
	\label{sec:Ls}
	
	In this appendix we give the definition of the fractional sublaplacians $\L^s$ and $\L_s$. Our exposition here will be concise---for more details we refer the reader to the book by G.B. Folland \cite{Folland2}, as well as \cite{FR,RT, RT2} for some short recap. Note that our notation and definitions may slightly differ from other present in the literature.
	
	After recalling the convolution between two functions $ f $ and $ g $ on $\Heis^d$
	$$
	(f*g)(z,t) := \int_{\Heis^d} f\left(z-z',t-t'\right)g(z',t') dz' dt',
	$$
	we define the $ \lambda$-twisted convolution as
	\begin{equation*}
		%\label{twist}
		(f^\lambda*_\lambda g^\lambda)(z)
		:=
		(f*g)^\lambda(z) 
		= \int_{\R^{2d}} f^\lambda(z-z')g^\lambda(z') e^{\frac{i}{2}\lambda(x'y-xy')} \,dz'
		,
	\end{equation*}
	where $ f^\lambda $ stands for the Euclidean inverse Fourier transform of $ f $ in the {vertical variable} $t$:
	\begin{equation*}
		%\label{eq:inverseFT}
		f^\lambda(z) := \int_{-\infty}^\infty f(z,t) e^{i\lambda t} \,dt.
	\end{equation*}
	It is worth commenting that the group Fourier transform of $f\in L^1(\mathbb H^d)$ is given by the Weyl transform $W_\lambda$ of $f^\lambda$. However, since for the purposes of this work we do not use the Fourier analysis on the Heisenberg group, again we refer the reader to \cite{Folland2} for this topic.

	Let us introduce also the scaled Laguerre functions of type $(d-1)$: 
	\begin{equation*}
		%\label{eq:Laguerre}
		\varphi_k^\lambda(z) := L_k^{d-1}\Big(\frac12 |\lambda||z|^2\Big)e^{-\frac14 |\lambda||z|^2},
	\end{equation*}
	where $L_k^{d-1} $ are the Laguerre polynomials of type $ (d-1)$ (see \cite[Chapter 1.4] {Than}). The functions $\{\varphi_k^\lambda\}_{k=0}^{\infty}$ form an orthogonal basis for the subspace consisting of radial functions in $ L^2(\R^{2d})$, namely functions invariant under the action of the group $\textbf{SO}(2d)$ of all orthogonal matrices of order $2d$ and determinant $1$.  
	In this way, we can write the so-called \textit{special Hermite expansion} of a function
	$ g $ on $ \R^{2d}$ as
	$$
	g(z)=(2\pi)^{-d} |\lambda|^d \sum_{k=0}^\infty  g*_\lambda \varphi_k^\lambda(z),
	$$
	and therefore give the spectral decomposition of the sublaplacian:
	\begin{equation*}
		%\label{eq:spectral}
		\mathcal{L}f(z,t) = (2\pi)^{-d-1} \int_{-\infty}^\infty \left[ \sum_{k=0}^\infty (2k+d)|\lambda|  (f^\lambda*_\lambda \varphi_k^\lambda)(z)\right] e^{-i\lambda t}  |\lambda|^d d\lambda.
	\end{equation*}
	The pure fractional sublaplacian $\mathcal{L}^s$ can be then defined for $s\ge0$
	via the spectral decomposition 
	\begin{equation}
		\label{def:Lsups0}
		\mathcal{L}^sf(z,t) = (2\pi)^{-d-1} \int_{-\infty}^\infty \left[ \sum_{k=0}^\infty \big((2k+d)|\lambda|\big)^s  (f^\lambda*_\lambda \varphi_k^\lambda)(z) \right] e^{-i\lambda t}  |\lambda|^d d\lambda,
	\end{equation}
	whereas the conformal fractional sublaplacian $ \mathcal{L}_s $ can be defined for $0\le s < d+1 $ by
	\begin{equation}
		\label{def:Lsubs0}
		\mathcal{L}_sf(z,t) = (2\pi)^{-d-1} \int_{-\infty}^\infty \left[ \sum_{k=0}^\infty (2|\lambda|)^s \frac{\Gamma(\frac{2k+d}{2}+\frac{1+s}{2})}{ \Gamma(\frac{2k+d}{2}+\frac{1-s}{2})} (f^\lambda*_\lambda \varphi_k^\lambda)(z) \right] e^{-i\lambda t}|\lambda|^d d\lambda.
	\end{equation} 

The operator $\mathcal{L}_s$ in \eqref{def:Lsubs0} was introduced by Branson, Fontana and Morpurgo \cite{BFM} in $\Heis^d$ (more precisely, see \cite[(1.33)]{BFM}), %, we also refer to \cite{BOO,FGMT, GL3,GL4, RT,RT2}{\color{blue}[H: Do we really need to refer these references ?]}.
 as the counterpart of the conformal fractional Laplacian from Riemannian geometry. 	The operator $\mathcal{L}_s$ occurs naturally in the context of CR geometry and scattering theory on the Heisenberg group: when we identify $\Heis^d$ as the boundary of the Siegel's upper half space in $ \C^{d+1}$, they have the important property of being conformally invariant (see for instance \cite{BFM}). The operator $\mathcal{L}_s$ differs completely from $\mathcal{L}^s$ since there is no geometry involved in the latter, and these operators only coincide in the limit as $s\to 1$, namely $ \mathcal{L}_1 = \mathcal{L}^1 = \mathcal{L} $. It is known that $\mathcal{L}_s$ also possesses an explicit fundamental solution of the form $c_{d,s}|(z,t)|_{\mathbb H}^{-Q+2s}$ with some constant $c_{d,s}>0$ (see for instance \cite[Section 3]{RT}), while an explicit expression for the fundamental solution for for $\mathcal{L}^s$ is not known.
	%	It can be proved that the operators $\mathcal{L}_s $ and $ \mathcal{L}^s $ are equivalent in $ L^p(\Heis^d) $, $1<p<\infty$, viz.\begin{equation}\label{eq:equiv}	c\|\mathcal{L}^sf\|_{L^p}\le \|\mathcal{L}_sf\|_{L^p}\le C \|\mathcal{L}^sf\|_{L^p}	\end{equation}	for some $c,C>0$.
	
	Equivalently we can also define $\L_s$ and $\L^s$ via the abstract spectral theorem, which comes handy for the purposes of this work. Using the formula \eqref{spectral_decomposition_theorem_1}, for $\L^s$ we can write
	\begin{equation*}
		\<\L^s f,g\> :=\int_0^\infty \lambda^s d\<E_\mathcal{\L}(\lambda)f,g\>,
		\qquad
		D(\L^s) :=\{f\in L^2(\mathbb H^d)\ |\ \int_0^\infty \lambda^{2s} d\<E_\mathcal{\L}(\lambda)f,f\><\infty\}.
	\end{equation*}
On the other hand, \eqref{def:Lsubs0} means that the operator $\mathcal{L}_s$ corresponds to the spectral multiplier
	\begin{equation*}
		%\label{eq:multiplier}
		(2|\lambda|)^s\frac{\Gamma\big(\frac{(2k+d)|\lambda)}{2|\lambda|}+\frac{1+s}{2}\big)}
		{\Gamma\big(\frac{(2k+d)|\lambda|}{2|\lambda|}+\frac{1-s}{2}\big)}, \quad k\in \N\cup\{0\},\ \lambda\in \R. 
	\end{equation*}
Thus, we can write (at least formally)
	\begin{equation}
		\label{def:Lsubs}
		\mathcal{L}_s:=(2|T|)^s\frac{\Gamma\big(\frac{\mathcal{L}}{2|T|}+\frac{1+s}{2}\big)}
		{\Gamma\big(\frac{\mathcal{L}}{2|T|}+\frac{1-s}{2}\big)}.
	\end{equation}
To define rigorously the RHS of this formula, we need to exploit the joint spectral theory  (see Appendix\til\ref{appendix_joint_spectral_theory} below) for $\L$ and $-iT$, using the fact that they strongly commute. Indeed, if we denote by $\F_T$ the Fourier transform in $t$ and by $\eta$ the Fourier variable in $t$, that is $\mathcal F_T^{-1}T\mathcal F_T=i\eta$, then \eqref{sublaplacian_2} yields
$
\L=\F_T^{-1}(-\partial_z^2-4i\eta N+\eta^2|z|^2)\mathcal F_T
$. 
Let $\widetilde{\mathcal L}(\eta):=-\partial_z^2-4i\eta N+\eta^2|z|^2$ for short. Since $\widetilde{\mathcal L}(\eta)$ and $\eta$ clearly commute strongly, we obtain
\begin{align*}
(\mathcal L-\sigma)^{-1}(-iT-w)^{-1}&=\mathcal F_T^{-1}(\widetilde{\mathcal L}(\eta)-\sigma)^{-1}(\eta-w)^{-1}\mathcal F_T=(-iT-w)^{-1}(\mathcal L-\sigma)^{-1}
\end{align*}
for any $\sigma,w\in \C\setminus\R$. Hence there exists a unique spectral measure $E_{\L,-iT}$ such that, for any measurable, a.e. finite function $\phi$ on $\R^2$, we can define the operator $\phi(\L,-iT)$ via the spectral theorem:
	\begin{align}\label{eq:jointLT}
		\<\phi(\L,-iT)f,g\>&:=\int_{\R^2}\phi(\lambda_1,\lambda_2)d\<E_{\L,-iT}(\lambda_1,\lambda_2)f,g\>, \\
		\notag
		D(\phi(\L,-iT)) &:=\left\{f\in L^2(\mathbb H^d)\ \left|\ \int_{\R^2}|\phi(\lambda_1,\lambda_2)|^2d\<E_{\L,-iT}(\lambda_1,\lambda_2)f,f\> <\infty\right.\right\}. 
	\end{align}
In particular, we can therefore define $\L_s$ choosing $\phi=\Phi_s$ in the above formula, with
\begin{equation*}\label{def:Phis}
	\Phi_s(\lambda_1,\lambda_2):=(2|\lambda_2|)^s\frac{\Gamma\big(\frac{\lambda_1}{2|\lambda_2|}+\frac{1+s}{2}\big)}{\Gamma\big(\frac{\lambda_1}{2|\lambda_2|}+\frac{1-s}{2}\big)},
\end{equation*}
getting again formula \eqref{def:Lsubs}.
%In other words, we can make sense of formula \eqref{def:Lsubs} equivalently via the group Fourier transform and the spectral decomposition, and via the spectral theorem.

Now, the three conditions \ref{H1}, \ref{H2} and \ref{H3} are easily verified for both $\L^s$ and $\L_s$. Indeed, the homogeneity can be checked directly by a change of variables in \eqref{def:Lsubs0} and \eqref{def:Lsups0}, after observing that 
%$\phi_k^\lambda(z)=\phi_^{\lambda/\delta^2}(\delta z)$ and 
$$ [(e^{i\tau A} f)^{\lambda} *_\lambda \varphi_k^\lambda](z) = e^{-(2d+2)\tau} [f^{\lambda/e^{2\tau}} *_{\lambda/e^{2\tau}} \varphi_k^{\lambda/e^{2\tau}}](e^\tau z). $$
Condition \ref{H2} for $\mathcal L_s$ is trivial. For $\L_s$ with $0<s<d+1$, Stirling’s formula 
$$
\Gamma(t)=\sqrt{\frac{2\pi}{t}}\left(\frac te\right)^t\left(1+O(t^{-1})\right),\quad t\to \infty,
$$
yields that the function
$$
m(r)=\frac{\Gamma\left(r+\frac{1+s}{2}\right)}{r^s\Gamma\left(r+\frac{1-s}{2}\right)}
$$
on $[d/2,\infty)$ satisfies $c_s\le m(r)\le C_s$ with some constants $c_s,C_s>0$. Using this fact with $r=\frac{2k+d}{2}$ and formulas \eqref{def:Lsups0} and \eqref{def:Lsubs0}, we obtain
$$
c_s\|\mathcal L^{s/2}f\|^2\le \|\mathcal L_s^{1/2}f\|^2\le C_s\|\mathcal L^{s/2}f\|^2
$$
since both $\mathcal L_s^{1/2}\mathcal L^{-s/2}$ and $\mathcal L^{-s/2}\mathcal L_s^{1/2}$ correspond to the spectral multiplier
$$
\sqrt{\frac{(2|\lambda|)^s\Gamma\big(\frac{2k+d}{2}+\frac{1+s}{2}\big)}{(2k+d)^s|\lambda|^s\Gamma\big(\frac{2k+d}{2}+\frac{1-s}{2}\big)}}=\sqrt{m\left(\frac{2k+d}{2}\right)}, \quad k\in \N\cup\{0\}.
$$
Condition \ref{H2} with $\H=\L_s$ thus holds. 
Condition \ref{H3} with $\H=\L^s$ is trivial. For $\H=\mathcal L_s$, condition \ref{H3} follows from Proposition \ref{proposition_joint_1} (2) in Appendix \ref{appendix_joint_spectral_theory} with $\varphi(\lambda_1,\lambda_2)= (\Phi_s(\lambda_1,\lambda_2)-\sigma)^{-1}$ and $\psi(\lambda_1)=(\lambda_1-w)^{-1}$. Equivalently, we can also see this from \eqref{def:Lsups0} and \eqref{def:Lsubs0}, since both $(\L_s-\sigma)^{-1}(\L-w)^{-1}$ and $(\L-w)^{-1}(\L_s-\sigma)^{-1}$ correspond to the spectral multiplier 
\begin{equation*}
	%\label{eq:multiplier}
	\left[ (2|\lambda|)^s\frac{\Gamma\big(\frac{2k+d}{2}+\frac{1+s}{2}\big)}
	{\Gamma\big(\frac{2k+d}{2}+\frac{1-s}{2}\big)}
	-\sigma
	\right]^{-1}
	[(2k+d)|\lambda|-w]^{-1}
	, \quad k\in \N\cup\{0\},\ \lambda \in \R. 
\end{equation*}
Moreover, $\L^s$ and $\L_s$ leave $\Ker N$ invariant, since they are invariant under rotations $z\mapsto Uz$ with $U\in \textbf{SO}(2d)$, and $N$ is a sum of generators of the rotations (see Remark\til\ref{rem:N}).

Finally, it is also worth mentioning that there is yet an another way to characterize the operator $\mathcal L_s$ though an extension problem (see \cite{FGMT,  RT2}).

\section{Joint spectral theory for commuting self-adjoint operators}
\label{appendix_joint_spectral_theory}

The purpose of this appendix is to record the joint spectral theory for two commuting self-adjoint operators. We refer to textbooks \cite{BiSo} for details. 

Let $\H_1$ and $\H_2$  be two self-adjoint operators on a Hilbert space $X$ such that they strongly commute: 
$$
[(\H_1-\sigma)^{-1},(\H_2-w)^{-1}]=0,\quad \sigma,w\in \C\setminus\R, 
$$
or equivalently $[E_{\H_1}(\Omega_1),E_{\H_2}(\Omega_2)]=0$ for any Borel measurable sets $\Omega_1,\Omega_2\in \mathcal B(\R)$. Then there exists a unique spectral measure $E_{\H_1,\H_2}(\cdot)$ on $\mathcal B(\R^2)$ such that $$E_{\H_1,\H_2}(\Omega_1\times\Omega_2)=E_{\H_1}(\Omega_1)E_{\H_2}(\Omega_2)$$ for any $\Omega_1,\Omega_2\in \mathcal B(\R)$ and that $E_{\H_1,\H_2}$ diagonalize $\H_1$ and $\H_2$ simultaneously, namely
$$
\<H_jf,g\>=\int_{\R^2}\lambda_j d\mu_{f,g},\quad f\in D(\H_j),\ g\in X,\ j=1,2,
$$
where $\mu_{f,g}=\<E_{\H_1,\H_2}(\lambda_1,\lambda_2)f,g\>$ (see \cite[Theorems 5.2.6 and 5.5.1]{BiSo}). For measurable $\varphi:\R^2\to \C\cup\{\infty\}$ satisfying $|\varphi|(\lambda)<\infty$ for a.e. $\lambda\in \R^2$, we then define the operator $\varphi(\H_1,\H_2)$ by
\begin{align*}
\<\varphi(\H_1,\H_2)f,g\>&:=\int_{\R^2}\varphi(\lambda_1,\lambda_2)d\mu_{f,g},\quad f,g\in D(\varphi(\H_1,\H_2)), \\
D(\varphi(\H_1,\H_2))&:=\left\{f\in X\ \left|\ \int_{\R^2}|\varphi(\lambda_1,\lambda_2)|^2d\mu_{f,f}<\infty\right.\right\}. 
\end{align*}

In what follows, we put $J_\varphi=\varphi(\H_1,\H_2)$ for short. 

%proposition
\begin{proposition}
\label{proposition_joint_1}
$J_\varphi$ is a densely defined closed operator on $X$ satisfying the following properties. 
\begin{itemize}
\item[(1)] For $\alpha,\beta\in \C$, $D(\alpha J_\varphi+\beta J_\psi)=D(J_\varphi)\cap D(J_\psi)$ and 
$\overline{\alpha J_\varphi+\beta J_\psi}= J_{\alpha\varphi+\beta\psi}
$. 
In particular, if one of $\varphi$ and $\psi$ belong to $L^\infty(\R^2)$, then ${\alpha J_\varphi+\beta J_\psi}= J_{\alpha\varphi+\beta\psi}$. 
\item[(2)] $D(J_{\varphi}J_{\psi})=D(J_{\varphi\psi})\cap D(J_{\psi})$ and $\overline{J_{\varphi}J_{\psi}}= J_{\varphi\psi}$. In particular, if $\psi\in L^\infty(\R^2)$, then $J_{\varphi}J_{\psi}= J_{\varphi\psi}$. 
\item[(3)] $J_{\overline{\varphi}}=(J_{\varphi})^*$, where $\overline{\varphi}$ is the complex conjugate of $\varphi$. 
\end{itemize}
\end{proposition}

\begin{proof}
The proof is completely the same as of the single case $\varphi(\H)$. See \cite[Theorems 5.4.3--5.4.7]{BiSo} for the proof of the single case. 
\end{proof}

%%%%%%%%%%%%%%%%%%%%%%%%%%%%%%%%%%%%%%%%%%%%%%%%%
%%%%%%%%%%%%%%%%%%%% Acknowledgement %%%%%%%%%%%%%%%%%%%
%%%%%%%%%%%%%%%%%%%%%%%%%%%%%%%%%%%%%%%%%%%%%%%%%
	
\section*{Acknowledgement}

L. Fanelli, L. Roncal and N. M. Schiavone 
are partially supported 
by the Basque Government through the BERC 2022--2025 program 
and 
by the Spanish Agencia Estatal de Investigaci\'{o}n
%State Research Agency 
through BCAM Severo Ochoa excellence accreditation CEX2021-001142-S/MCIN/AEI/10.13039/501100011033.

L. Fanelli is also supported by the projects PID2021-123034NB-I00/MCIN/AEI/10.13039/501100011033 funded by the Agencia Estatal de Investigaci\'{o}n, 
IT1615-22 funded by the Basque Government, 
and by Ikerbasque.

H. Mizutani is partially
supported by JSPS KAKENHI Grant Numbers JP21K03325 and JP24K00529.

L. Roncal is also supported by the projects
RYC2018-025477-I and CNS2023-143893, funded by Agencia Estatal de Investigaci\'on, and PID2023-146646NB-I00 funded by MICIU/AEI/10.13039/501100011033 and by ESF+, and IKERBASQUE.

N. M. Schiavone was also supported by the  grant FJC2021-046835-I funded by the EU \lq\lq NextGenerationEU''/PRTR and by MCIN/AEI/10.13039/501100011033,
and
by the project PID2021-123034NB-I00/MCIN/ AEI/10.13039/501100011033 funded by the Agencia Estatal de Investigaci\'{o}n.
He is also member of the \lq\lq Gruppo Nazionale per L'Analisi Matematica, la Probabilit\`{a} e le loro Applicazioni'' (GNAMPA) of the \lq\lq Istituto Nazionale di Alta Matematica'' (INdAM).

%%%%%%%%%%%%%%%%%%%%%%%%%%%%%%%%%%%%%%%%%%%%%%%%
%%%%%%%%%%%%%%%%%%%%% Bibliography %%%%%%%%%%%%%%%%%%%%
%%%%%%%%%%%%%%%%%%%%%%%%%%%%%%%%%%%%%%%%%%%%%%%%

\end{document}